\documentclass[11pt]{amsart}

\usepackage{amscd}
\usepackage{appendix}
\usepackage{amssymb}
\usepackage{amsfonts}
\usepackage{amsrefs}
\usepackage{enumerate}
\usepackage{color}
\usepackage[T1]{fontenc}
\usepackage{mathrsfs} 

\numberwithin{equation}{section}

\title[Local integrability for large $p$]
{Local integrability results in harmonic analysis on reductive
groups in large positive characteristic}

\author{Raf Cluckers, Julia Gordon and Immanuel Halupczok}
\newtheorem{thm}{Theorem}[section]
\newtheorem{prop}[thm]{Proposition}
\newtheorem{cor}[thm]{Corollary}
\newtheorem{lem}[thm]{Lemma}

\theoremstyle{definition}
\newtheorem{defn}[thm]{Definition}

\theoremstyle{remark}

\newtheorem{rem}[thm]{Remark}

\newcommand{\C}{{\mathbb C}}
\newcommand{\K}{{\mathbb K}}
\newcommand{\Z}{{\mathbb Z}}
\newcommand{\Q}{{\mathbb Q}}
\newcommand{\N}{{\mathbb N}}

\newcommand{\R}{{\mathbb R}}
\newcommand{\rf}{k}
\newcommand{\ri}{\Omega}
\newcommand{\gal}{\operatorname{Gal}}

\newcommand{\unr}{\mathrm{unr}}
\newcommand{\orb}{\mathrm{orb}}
\newcommand{\nilp}{\mathrm{nilp}}

\newcommand{\de}{{\text{Def}}}
\newcommand{\rde}{{\text{RDef}}}
\newcommand{\cA}{{\mathcal A}}
\newcommand{\cB}{{\mathcal B}}

\newcommand{\lef}{{\mathbb L}}
\newcommand{\cF}{{\mathcal F}}
\newcommand{\cC}{{\mathscr C}}

\newcommand{\scD}{{\mathscr D}}
\newcommand{\cL}{{\mathcal L}}
\newcommand{\ord}{\operatorname{ord}}
\newcommand{\ac}{\overline{\operatorname{ac}}}

\newcommand{\bG}{{\bf G}}
\newcommand{\bS}{{\bf S}}
\newcommand{\bT}{{\bf T}}

\newcommand{\bM}{{\bf M}}

\newcommand{\bZ}{{\bf Z}}

\newcommand{\GL}{\operatorname{GL}}
\newcommand{\SL}{\operatorname{SL}}
\newcommand{\gl}{\mathfrak{gl}}
\newcommand{\fm}{\mathfrak{m}}
\newcommand{\mexp}{\mathrm {\bf e}}

\newcommand{\fg}{{\mathfrak g}}

\newcommand{\ft}{{\mathfrak t}}
\newcommand{\fu}{{\mathfrak u}}
\newcommand{\res}{\operatorname{res}}

\newcommand{\sem}{\mathrm{ss}}
\newcommand{\spl}{\mathrm{spl}}

\newcommand{\reg}{\mathrm{reg}}
\newcommand{\cD}{{\mathcal D}}
\newcommand{\tf}{{C_c^\infty}}
\newcommand{\cN}{{\mathcal N}}
\newcommand{\Ad}{\operatorname{Ad}}
\newcommand{\ad}{\operatorname{ad}}

\newcommand{\scB}{\mathscr{B}}
\newcommand{\scA}{\mathscr{A}}

\newcommand{\cO}{{\mathcal O}}

\newcommand{\fto}{{\widehat\mu_{\cO}}}
\newcommand{\tr}{\operatorname{Tr}}

\newcommand{\aut}{\operatorname{Aut}}
\def\llp{\mathopen{(\!(}}
\def\llb{\mathopen{[\![}}
\def\rrp{\mathopen{)\!)}}
\def\rrb{\mathopen{]\!]}}

\begin{document}

\begin{abstract}
Let $\bG$ be a connected reductive algebraic
group over a non-Archimedean local field $\K$, and let $\fg$ be its Lie algebra.
By a theorem of Harish-Chandra, if $\K$ has characteristic zero, the
Fourier transforms of orbital integrals are represented on the
set of regular elements in $\fg(\K)$ by locally constant functions, which, extended by zero to all of $\fg(\K)$, are locally integrable.
In this paper, we prove that
these functions are in fact specializations of constructible motivic exponential functions. Combining this with the Transfer Principle for integrability of \cite{TI}, we obtain that Harish-Chandra's theorem holds also when $\K$ is a non-Archimedean local  field of sufficiently large positive characteristic.
{Under the hypothesis that mock exponential map exists}, this also implies local integrability  of
Harish-Chandra characters of admissible representations of $\bG(\K)$, where
$\K$ is an equicharacteristic field of sufficiently large
(depending on the root datum of $\bG$) characteristic.




\section*{R\'esum\'e}
Soit $\mathbf{G}$ un groupe alg\'ebrique reductif connexe au dessus d'un
corps local non-archim\'edien $\K$, et soit $\fg$ son alg\`ebre de Lie.
D'apr\`es un th\'eor\`eme de Harish-Chandra, si $\K$ est de
caract\'eristique z\'ero, alors les transform\'es de Fourier
d'int\'egrales orbitales sont repr\'esent\'es, sur l'ensemble des
\'elements r\'eguliers de $\fg(\K)$, par des fonctions localement
constantes, qui, si on les \'etend par z\'ero \`a tout $\fg(\K)$, sont
localement int\'egrables. Dans ce papier, nous d\'emontrons que
ces fonctions sont en fait des sp\'ecialisations de fonctions
motiviques constructibles exponentielles. En combinant ceci avec le
principe de transfert d'int\'egrabilit\'e de \cite{TI}, nous obtenons
que le th\'eor\`eme de Harish-Chandra est valable aussi quand $\K$ est
un corps local non-archim\'edien de caract\'eristique positive
suffisamment grande. Sous l' hypoth\`ese que l'application exponentielle
feinte existe, ceci implique aussi l'int\'egrabilit\'e  locale des
caract\`eres de Harish-Chandra de repr\'esentations admissibles de
$\bG(\K)$, o\`u $\mathbb{K}$ est un corps d'equicaract\'eristique
suffisamment grande (en fonction de la donn\'ee radicielle de $\bG$).

\end{abstract}

\maketitle

\section{Introduction}

In this paper we prove an extension  of Harish-Chandra's
theorems about local integrability of the functions representing various
distributions arising in harmonic analysis on $p$-adic groups  to the
positive characteristic case, when the residue characteristic is large.
Our method consists in transferring Harish-Chandra's results from characteristic zero to positive
characteristic.
In the recent years such transfer has become a prominent technique, culminating in the transfer of the Fundamental
Lemma from positive characteristic to characteristic zero,
\cite{cluckers-hales-loeser},  \cite{waldspurger:transferFL}. Two distinct ways of carrying out transfer have been described in the literature -- one method is based on the idea of close local fields, due to D. Kazhdan and  J.-L. Waldspurger, cf. \cite{waldspurger:weighted}. The other method is based on the program
outlined by T.C. Hales in \cite{hales:computed} of making harmonic
analysis on reductive groups over non-Archimedean local fields ``field-independent'' via the use of motivic integration, and this is the method we use.

We observe that the statements we are proving in this paper are much more analytic in nature than any of the statements previously handled by the transfer methods -- namely, here we talk about $L^1$-integrability, as opposed to much more algebraic-type statements about equalities between integrals of functions that are known to be integrable. In this sense it is somewhat surprising that the transfer is still possible, and it requires a new type of  transfer principle, which we prove in  \cite{TI}. We note that the use of this very general transfer principle allows us to avoid substantial technical difficulties that one faces when using the method of transfer based on the technique of close local fields, at the cost, however, of not getting a precise lower bound on the characteristic of the fields for which our results apply.


Our main technical result is
Theorem \ref{thm:ssint} showing that the functions representing the Fourier transforms of the orbital integrals form a family of so-called constructible motivic exponential functions.
These functions were introduced by
R. Cluckers and F. Loeser in \cite{cluckers-loeser:fourier}; they are defined in a field-independent manner by means of logic  (in fact, we use a slight generalization; see 
\S\ref{subsub:gen}).
Theorem \ref{thm:ssint} implies that Transfer principles for integrability and boundedness
apply to the Fourier transforms of all orbital integrals, and in particular, to the nilpotent ones. Once all the required properties of the nilpotent orbital integrals are transferred to the positive characteristic in Theorem \ref{thm:orb int loc int},
the analogues of many of the classical results for general distributions follow, thanks to the work of DeBacker \cite{debacker:homogeneity}, and J. Adler and J. Korman, \cite{adler-korman:loc-char-exp}. Thus we  obtain our main results: Theorems \ref{thm:char} and \ref{thm:general} (the former assumes the hypothesis on the existence of a mock exponential map, which we review in \ref{subsub:exp}).

We note that for $\GL_n$,
the local integrability of characters was
proved by Rodier \cite{rodier:loc-int}
for $p>n$, and by a different method, by
B. Lemaire \cite{lemaire:gl_n} for arbitrary $p$.
Lemaire also proved the local integrability of characters for
the inner forms of $\GL_n$ and $\SL_n$, and for twisted characters of $\GL_n$,
\cite{lemaire:gl_n(D)}, \cite{lemaire:sl_n(D)}.

{\bf Acknowledgement. } We are indebted to Thomas Hales,
Jonathan Korman, and Jyotsna Diwadkar, without whose influence this work
could not have appeared in the present form.
The second-named author is very grateful to
Jeff Adler,  William Casselman, Clifton Cunningham, Fiona Murnaghan, Loren Spice,
Sug Woo Shin, and Nicolas Templier
for multiple helpful
communications, and to Lance Robson for a careful reading of parts of the article.
{We thank the referee for  the careful reading and multiple helpful comments and corrections.
We gratefully acknowledge the support by several granting agencies, in particular, the grant G.0415.10 of the Fund for Scientific Research - Flanders - Belgium, the SFB~878 of the Deutsche
Forschungsgemeinschaft; the European Research Council under the European Community's Seventh Framework Programme (FP7/2007-2013) / ERC Grant Agreement No. 246903 NMNAG, the Labex CEMPI  (ANR-11-LABX-0007-01), and NSERC of Canada}.
\section{Results}

\subsection{Notation}\label{sub:notation}
For a discretely valued field $\K$,  its ring of integers will be denoted by
$\ri_\K$, the maximal ideal by ${\mathfrak p}_\K$, and the
residue field  by $\rf_\K$.

Let $\cA$  be the collection of all non-Archimedean local fields $\K$
of
characteristic zero,
with a chosen uniformizer $\varpi_\K$ of $\ri_\K$, and
let $\cB$ be the collection of all local fields $\K$ of positive
characteristic,  with a
uniformizer $\varpi_\K$ of $\ri_\K$.
The notation
$\K$ will always stand for a local field that lies
in $\cA\cup \cB$.
For an integer $M>0$,
we will also often use the collections $\cA_{M}$ and
$\cB_{M}$ of fields in $\cA$ and $\cB$ respectively,
with residue characteristic greater than $M$.

We use Denef-Pas language $\cL_\Z$ with coefficients in $\Z$ -- this is a first-order language of logic; roughly speaking, formulas in this language define subsets of affine spaces uniformly over all local fields $\K\in \cA\cup \cB$,
(see Appendix B for precise definitions).
By  ``definable'' we shall  mean, definable in the language $\cL_\Z$.
We survey all the definitions and theorems from the theory of motivic integration that we use in Appendix B. {We note that if one wishes to work only with reductive groups defined over a fixed number field $E$ (with a ring of integers $\ri$) and its completions, then one can use the language $\cL_\ri$ defined in Appendix B; all the results still apply since any
language $\cL_\ri$   includes the language $\cL_\Z$.}

Throughout this paper, $\bG$ stands for a connected reductive algebraic group over a local field $\K$,
and $\fg$  for its Lie algebra. For $X\in \fg(\K)$, $D_G(X)$ is the discriminant of $X$,
see Appendix A for the definition.

Following Kottwitz, \cite{kottwitz:clay}, we call a function $F(X)$, defined and locally constant on the set of regular elements $\fg(\K)^\reg$, ``nice''
if it satisfies the following two requirements:
\begin{itemize}
\item when extended by zero to
all of $\fg(\K)$, it is locally integrable, and
\item the function $|D_G(X)|^{1/2}F(X)$ is locally bounded on $\fg(\K)$.
\end{itemize}

Similarly, call a function on $\bG(\K)$ ``nice'', if it satisfies
the same conditions on $\bG(\K)$, with $D_G(X)$ replaced by its group version $D_G(g)$, namely,
the coefficient at $t^r$ (where $r$ is the rank of $\bG$) in the polynomial $\det\left((t+1)I-\Ad(g)\right)$.

\subsection{The statements}
We refer to Appendix A for all the definitions (of orbital integrals, etc.) and a survey of
the classical results.

Our main result states that the Fourier transforms of orbital integrals are represented by nice functions, in large positive characteristic.

\begin{thm}\label{thm:orb int loc int}
There exists a constant $M_\bG^\orb >0$ that depends only on the  absolute root datum
of $\bG$, such that
for every $\K\in \cB_{M_\bG^\orb}$,  for every $X\in \fg(\K)$,
the function $\widehat \mu_X$ is a nice function on $\fg(\K)$.
\end{thm}

In this theorem and all similarly phrased  statements below, our assertion that there exists a constant $M>0$ that depends only on the  absolute root datum
of $\bG$ such that so-and-so properties hold for $\bG(\K)$ with $\K\in \cA_M\cup \cB_M$, has the following meaning.
As discussed in \S\ref{sub:roots} below, given an absolute root datum $\Psi$ (which is a field-independent construct), there exist finitely many possibilities  for the root data of reductive groups over non-Archimedean local fields having the absolute root datum $\Psi$. We parameterise these possibilities by points of a definable set in \S\ref{sub:roots}.
Then our statement says that there exists a constant $M$ that depends only on $\Psi$, such that for every local field
$\K\in \cA_M\cup \cB_M$, for all possible connected reductive groups  $\bG$ defined over $\K$ with absolute root datum $\Psi$, the assertions of the theorem hold.

Theorem \ref{thm:orb int loc int} is proved below in \S\ref{subsub:th1p2}.

Thanks to the local character expansion near a tame semisimple element, the above theorem implies that
Harish-Chandra characters of admissible representations are represented by nice functions on the group, under {the additional  hypothesis on the existence of a so-called mock exponential map}.
Local character expansion in large positive characteristic
is proved by DeBacker \cite{debacker:homogeneity} near the identity,
and by Adler-Korman \cite{adler-korman:loc-char-exp} near a general tame semisimple element, if the mock exponential map exists.
We start by quoting the hypothesis, which uses the notation defined in \S \ref{sub:MP} below.

\subsubsection{The exponential map hypothesis}\label{subsub:exp}
\cite{debacker:homogeneity}*{Hypothesis 3.2.1}.
{\it Suppose $r>0$. There exists a bijective map $\mexp:\fg(\K)_r\to\bG(\K)_r$
such that
\begin{enumerate}
\item for all pairs $x\in \scB(\bG, \K)$, $s\in \R_{\ge r}$, we have
\begin{enumerate}
\item $\mexp(\fg(\K)_{x,s})=\bG(\K)_{x,s}$,
\item For all $X\in \fg(\K)_{x,r}$ and for all $Y\in \fg(\K)_{x,s}$, we have
$\mexp(X)\mexp(Y)\equiv \mexp(X+Y) \mod \bG(\K)_{x, s^+}$,
and
\item $\mexp$ induces a group isomorphism of $\fg(\K)_{x,s}/\fg(\K)_{s, s^+}$ with
 $\bG(\K)_{x,s}/\bG(\K)_{s, s^+}$;
\end{enumerate}
\item for all $g \in \bG(\K)$ we have
$\operatorname{Int}(g)\circ \mexp=\mexp\circ\operatorname{Ad}(g)$;
\item $\mexp$ carries $dX$ into $dg$  (where $dX$ and $dg$ are Haar measures on $\fg(\K)$ and $\bG(\K)$, respectively, associated with the same normalization of the Haar measure
on $\K$,  cf.\S \ref{sub:HM}).
\end{enumerate}
}

For classical groups one can take $\mexp$ to be the Cayley transform,
for all $r>0$.

\begin{thm}\label{thm:char}
There exists a constant $M_\bG >0$ that depends only on the absolute root datum of $\bG$,
such that if $\K\in \cB_{M_\bG}$ and
Hypothesis
\ref{subsub:exp} holds for $\bG(\K)$ with some $r>0$, then
for every  admissible
representation  $\pi$ of ${\bG}(\K)$,
its Harish-Chandra character $\theta_{\pi}$ is a nice function on
$\bG(\K)$; in particular, the integral
$\int_{\bG(\K)}\theta_{\pi}(g)f(g)\, dg $
converges, and equals $\Theta_\pi(f)$, for all test functions $f\in C_c^\infty(\bG(\K))$.
\end{thm}

{We prove this theorem in \S \ref{subsub:th2} below. }

\begin{rem}
DeBacker's result on the local character expansion that we use in the proof of this theorem  requires, in its full strength, the assumption that Hypothesis \ref{subsub:exp} holds for
$r\in \R$ such that
$\fg_r=\fg_{\rho(\pi)^+}$, where $\rho(\pi)$ is the depth of $\pi$. Here we only use the fact that the local character expansion holds in some (definable) neighbourhood of the identity,  which is yielded by DeBacker's proof
assuming just the existence of the mock exponential for some $r>0$. Note also that we do not require the mock exponential  map to be definable.
\end{rem}

Finally, Theorem \ref{thm:orb int loc int} also implies (thanks to a result of DeBacker) that Fourier transforms of general invariant distributions on $\fg(\K)$ with support bounded modulo conjugation are represented by nice functions  in a neighbourhood of the origin. This is Theorem \ref{thm:general}.

The rest of the main body of the paper is devoted to the proof of these theorems. Two appendices are provided for the reader's convenience -- Appendix A contains a brief summary of the definitions and relevant classical results in harmonic analysis on $p$-adic groups, and Appendix B summarises the definitions and results from the theory of motivic integration, which is used in the proofs.

\section{Definability of Moy-Prasad filtration subgroups}
From now on we will freely use the language of definable subassignments, and constructible motivic functions; please see Appendix B for definitions and all related notation.
We start by setting up the definition of the group, Lie algebra, and Moy-Prasad filtration subgroups in Denef-Pas language.
As explained in Appendix B, for a group,  specific subgroup, etc. to be definable, roughly speaking, means that it can be defined uniformly for all local fields $\K$ of sufficiently large residue characteristic, by formulas in a first-order language of logic, which do not depend on the field themselves.

\subsection{Root datum and the group}\label{sub:roots}
The first step is to realize the group and its Lie algebra as definable subassignments, so that the methods of motivic integration apply.
From now on we assume that the residue characteristic $p$ is large enough so that
the group $\bG$ splits over a tamely ramified extension of $\K$.

Split reductive groups $\bG$ are classified by the root data
$\Psi=(X_{\ast}, \Phi, X^{\ast}, \Phi^\vee)$ consisting of the character group of
a split maximal torus $\bT$ in $\bG$, the set of roots, the cocharacter group, and the set of coroots.
The set of possible root data of this form  is completely field-independent.
Given a root datum $\Psi$, the group $\bG(F)$ is a definable subset of $\GL_n(F)$,
defined as the image of a definable embedding $\Xi:\bG\hookrightarrow \GL_n$, defined over {$\Z[1/R]$ for some large enough $R$} (see \cite{cluckers-hales-loeser}*{\S 4.1}, where such an embedding is denoted by $\rho_D$, with $D$ denoting the root datum).

We showed that general reductive groups are definable (or, more precisely, appear as members of a constructible family),  in \cite{S-T}*{Appendix B}.
This is based on the fact that
every reductive group splits over the separable closure of $F$, and
the $F$-forms of a group are in one-to-one correspondence with the Galois cohomology set $H^1(F, \mathbf{Aut}(\bG))$
(see e.g. \cite{springer:lag2}*{\S 16.4.3}).
Here we  {recall this construction briefly, also introducing the notation for the
intermediate unramified extension of $\K$ that will be used below}.

Recall that we are assuming that $p$ is large enough so that
$\bG$ splits over a tamely ramified extension; let $e$ be the ramification index. Then there exists an unramified  extension
$\K_f/\K$ of some degree $f$, such that $\bG$ splits over a field
 $L$, which is a totally ramified Galois extension of $\K_f$.
Let $\theta$ be a generator of $\gal(\K_f/\K)$ (the Frobenius element); and let
 $m=fe$ be the degree $[L:\K]$.
Let  {$\Gamma=\gal(L/\K)=\{ \sigma_1, \dots \sigma_m \}$.}

 We have the exact sequence of Galois groups
$$
1 \to \gal(L/\K_f) \to \gal(L/\K) \to \gal(\K_f/\K) \to 1.
$$

 Let us assume that $\{ \sigma_1, \dots \sigma_e \}$ is the subgroup of $\Gamma$ fixing $\K_f$ and that $\sigma_m$ projects to $\theta$ under the last map.

 In  \cite{S-T}*{\S B.4.2}, we constructed  a definable subassignment
$S_{[\Gamma]} \subset h[ {m+m^3}, 0,0]$, with the following property.
Given a local field $\K$ of sufficiently large residue characteristic,
$S_{[\Gamma]}$ specializes to the set of tuples
$(\bar b, \sigma_1,\dots, \sigma_m)$, where:
\begin{itemize}
\item $\bar b$ is a tuple of coefficients of a minimal polynomial over $\K$ that gives rise to a degree $m$ extension, which we denote by $\K_{\bar b}$;
\item  $\sigma_1, \dots, \sigma_m$
are $m\times m$ matrices, defining automorphisms of
$\K_{\bar b}$ over $\K$, and
\item the group  $\{\sigma_1, \dots, \sigma_m \}$ is isomorphic to
$\Gamma$.
\end{itemize}
  We also can, and do, add the condition that $\K_{\bar b}$ contains an unramified extension
  $\K_f$ of degree $f$, fixed by  $\{ \sigma_1, \dots, \sigma_{e}\}$, and that
 $\sigma_m$ projects to  $\theta$ -- the Frobenius element of $\K_f$, by stipulating that the restriction of $\sigma_m$ is a generator of $\gal(\K_f/\K)$, which can be phrased using Denef-Pas language formulas.
 {We are using $[\Gamma]$ as a subscript (as opposed to $\Gamma$) to emphasize that the subassignments
$S_{[\Gamma]}$ and $Z_{[\Gamma]}$ depend only on the isomorphism class of $\Gamma$, and not on a specific group.}

Suppose $\Psi=(X_\ast, \Phi, X^\ast, \Phi^\vee)$ is an absolute root datum as above. Then it defines a split reductive group
$ {\bG^\spl}$ over $\K$, and therefore we get a definable subassignment $Z_{[\Gamma]}$ over $S_{[\Gamma]}$
that specializes to the set of $1$-cocycles
$Z^1(\Gamma, \aut(\mathbf{ {G^\spl}})(\K_{\bar b}))$.
Finally, suppose $\bG$ is a group defined over $\K$ that splits over an extension
$L$ as above.
Then there exists a tuple $\bar b$ such that $L$ is isomorphic to $\K_{\bar b}$.
Let $\bG^{\spl}$ be the split form of $\bG$. Then we can think of $\bG$ as the group
$\bG_z$ corresponding to a cocycle
$z\in Z^1(\Gamma, \aut(\mathbf{G}^\spl)(\K_{\bar b}))$. It follows that $\bG(\K)$ appears as a fibre of a definable subassignment over $Z_{[\Gamma]}$ (by taking
 {$\{ z\cdot\sigma_1, \dots,  z\cdot\sigma_m\}$}-fixed points, cf. \cite{S-T}*{\S B.4.3}).

Given an absolute root datum $\Psi$, there are finitely many possibilities for the root data of the groups $\bG$ over $\K$ with the absolute root datum $\Psi$. Let $M_{\Psi}$ be the constant such that when $\K\in \cA_{M_\Psi}\cup \cB_{M_\Psi}$, all possible reductive groups
$\bG$ over $\K$ with the absolute root datum $\Psi$ (up to isomorphism) appear as fibres of definable subassignments over the subassignments  $Z_{[\Gamma]}$, as $[\Gamma]$ runs over the finite set of all the possibilities relevant for $\Psi$.

\subsection{Bruhat-Tits building}\label{sub:building}
Here we follow the notation of \cite{debacker:nilp} and
\cite{adler-debacker:bt-lie}
as much as possible.
Let us first review this notation.
Let $\scB=\scB(\bG, \K)$ denote the (enlarged) building of $\bG(\K)$.
Fix a maximal unramified extension $\K^\unr$ of $\K$.
Let $\bS$ be a maximal $\K$-split torus of $\bG$. Let $\bT$ be the maximal
$\K^\unr$-split torus of $\bG$ containing $\bS$.  Let $\bZ$ be the centralizer of $\bT$ in $\bG$; it is a maximal torus of $\bG$, defined over $\K$.
Let $L$ be the extension over which $\bG$ splits, as above, and let
$\K_f=\K^\unr\cap L$ be its maximal unramified part. Then $\bT$ splits over
$\K_f$.
Let $\scA$ be the apartment of $\bT(\K_f)$ in $\scB(\bG, \K_f)$. We can identify
$\scA(\bS, \K)$ with the $\gal(\K_f/\K)$-fixed points of $\scA$.

Let $\Phi^\unr$ be the set of roots of $\bG$ relative to $\bT$ and $\K_f$, and let $\tilde\Phi^\unr$ be the set of affine roots of $\bG$ relative to $\bT$, $\K_f$, and our choice of valuation
on $\K$.
We observe that $\Phi^\unr$ can be recovered from the root datum and the action $z\cdot \sigma_i$,  {$1\le i \le m$}, (where $z$ and $\sigma_i$ are as above in \S\ref{sub:roots}).
Hence, we can use $\Phi^\unr$ and $\tilde\Phi^\unr$ in the constructions of subassignments
over $S_{[\Gamma]}$. In this sense there is no harm in including the (possibly non-reduced) root system $\Phi^\unr$ as part (though redundant) of the given root datum defining the group $\bG$.

In this paper, we will only need to use
a fixed alcove $C$ in the apartment
$\scA$ such that
$\gal(\K_f/\K)$-fixed points of its closure $\bar C$ (in the $p$-adic topology) contain an alcove of $\scB$.
Note that $\bar C$ is a poly-simplicial set.
Moreover, the set  $\bar C^{z\cdot\theta}$
of  $\gal(\K_f/\K)$-fixed points of
$\bar C$ is also a poly-simplicial set, since the Galois action is compatible with the poly-simplicial structure.
In fact, we will need only the following information
about the set  $\bar C^{z\cdot\theta}$:
\begin{enumerate}
\item\label{c1} the list of its faces;
\item\label{c2} incidence relations between the faces;
\item\label{c4} a certain finite set of points in $\bar C$, called {\it optimal points}, discussed in the next subsection.
\end{enumerate}

We observe that $\scA=X_{\ast}(\bT)\otimes \R$ is an affine space of {dimension} determined by $\Phi^\unr$, and the affine roots (which also are pre-computed from $\Phi^\unr$)  define the hyperplanes in it, which, in turn, determine $\bar C$.
Thus, the list of faces of $\bar C$ can be pre-computed once the root system $\Phi^\unr$
is given. The action of $z\cdot \theta$ determines a permutation $\tau$ of $\Phi^\unr$, which, in turn, allows us to determine the list of faces of $\bar C^{z\cdot \theta}$.
In summary, the information we need about
$\bar C^{z\cdot\theta}$ is determined by the root datum $\Psi$ and the permutation
$\tau$; and $\tau$ is determined by the parameter $z\in Z_{[\Gamma]}$ in a definable way.
More precisely, given the root datum $\Psi$, there is a finite number
of possibilities for the list of faces of $\bar C^{z\cdot\theta}$, and we can decompose $Z_{[\Gamma]}$ into a disjoint union of  finitely many definable subsets, according to which possibility of
$\bar C^{z\cdot\theta}$ a given cocycle $z$ gives rise to.
We will denote these subsets, indexed by the pairs $(\Phi^\unr, \tau)$,
where $\Phi^\unr$ is a root system and $\tau$ is a permutation acting
on $\Phi^\unr$, by
$Z_{\Phi^\unr, \tau}$.
Once we have done that, we can assume that the list of faces of $\bar C^{z\cdot\theta}$
is part of the data defining $\bG$, and use it in the definitions of definable sets with parameters in $Z_{[\Gamma]}$.

Now let us turn to the set of optimal points. We will see that it can also
be pre-computed from the root datum.

\subsubsection{Optimal points}\label{subsub:OP}
In \cite{moy-prasad:k-types}*{\S 6.1}, Moy and Prasad
define the set $\mathcal O$ of the so-called optimal points; we will denote this set by $\mathfrak P$, since the notation $\mathcal O$ is reserved for the
orbits.

Let $C$ be the alcove in $\scA$ that gave rise to the set $\bar C^{z\cdot\theta}$ as above.
Let $\Sigma$ be the set of affine roots $\psi\in\tilde \Phi^\unr$ that satisfy
$\psi\vert_{C}>0$, $(\psi-1)\vert_{C}<0$. This is a finite set that depends only
on $\Phi^\unr$.
Further, {let ${\mathfrak C}_{\Sigma}$ be the collection of all the
$\gal(\K_f/\K)$-invariant subsets $\mathfrak S$ of $\Sigma$; this finite}
collection depends only on the root datum $\Phi^\unr$ and the permutation $\tau$, as above.

Let $\mathfrak S\subset \Sigma$ be an element of ${\mathfrak C}_{\Sigma}$.
Now we quote \cite{adler-debacker:bt-lie}*{\S 2.3}, where it is
stated that there exists a point
$x_{\mathfrak S}\in \bar{C}$ such that:
\begin{enumerate}
\item[(i)] $\min_{\psi\in \mathfrak S}\psi(x_{\mathfrak S})\ge
\min_{\psi\in \mathfrak S}\psi(y)$ for all $y\in \bar{C}$;
\item[(ii)] $\psi(x_{\mathfrak S})$ is rational for all $\psi\in {\tilde\Phi^\unr}$;
\item[(iii)] $x_{\mathfrak S}$ is  $\gal(\K_f/\K)$-invariant.
\end{enumerate}

We observe that for the future constructions, we do not need the point
$x_{\mathfrak S}$ itself, but rather the tuple of its ``baricentric coordinates''
$\left(\psi(x_{\mathfrak S})\right)_{\psi\in \Sigma}$.
As pointed out in \cite{moy-prasad:k-types}*{\S 6.1}, finding optimal
points is a problem of linear programming. The input for this problem
is the field-independent set of affine roots $\Sigma$; thus the output is also a field-independent tuple of rational coordinates $\psi(x_{\mathfrak S})$.

We denote by $\mathfrak P_{\Phi^\unr, \tau}$ the set
$$\mathfrak P_{\Phi^\unr, \tau}=\{(\psi(x_{\mathfrak S}))_{\psi\in \Sigma}\}_{\mathfrak S\in \mathfrak C_\Sigma}.$$

\subsection{Moy-Prasad filtrations}\label{sub:MP}
In \cite{moy-prasad:k-types}, Moy and Prasad associate with each pair
$(x, r)$, where $x\in \scB(\bG,\K)$ and $r\ge 0$,
(respectively, $r\in \R$):
\begin{itemize}
\item subgroups $\bG(\K)_{x, r^+}\subset \bG(\K)_{x, r}$ of $\bG(\K)$,  for $r\ge 0$;
\item lattices $\fg(\K)_{x, r^+}\subset \fg(\K)_{x, r}$ in $\fg(\K)$,  for $r\in \R$.
\end{itemize}

When $r=0$, it is omitted from the notation; thus by definition,
$\bG(\K)_{x}=\bG(\K)_{x, 0}$, $\bG(\K)_{x}^+=\bG(\K)_{x, 0^+}$,
$\fg(\K)_x=\fg(\K)_{x, 0}$, $\fg(\K)_x^+=\fg(\K)_{x, 0^+}$.
The groups $\bG(\K)_{x}$  and $\bG(\K)_{x}^+$ and the corresponding
lattices in the Lie algebra  depend only on the facet that contains the
point $x$. Therefore, for a facet $F$ we will denote them by
$\bG(\K)_{F}$, $\bG(\K)_{F}^+$, and $\fg(\K)_F$, $\fg(\K)_F^+$, respectively.

We will use the fact that for a group that splits over a
tamely ramified extension, the filtration subgroups with $r>0$ can
be obtained from {its split form} by taking Galois-fixed points.
We first recall the definitions (this version is quoted from
\cite{adler-debacker:bt-lie}, see also \cite{debacker:singapore}) for the split group
$\bG^\spl(L)=\bG(L)$, where $L$ is the extension that splits $\bG$, as above.
First, for  any torus $\bT$ defined over $L$, and for any extension $E$ of
$L$, define, for any $r\in \R$,
$$\ft(E)_r:=\{H\in \ft(E)\mid \ord(d\chi(H))\ge r \text{ for all }
\chi \in X^\ast(\bT)\}.$$
Fot a torus $\bT$ that is split over $L$, one can  define the
filtration subgroups of $\bT(E)$ simply as follows:
for $r\ge 0$, let
$$\bT(E)_r:=\{t\in \bT(E)\mid \ord(\chi(t)-1)\ge r \text{ for all }
\chi \in X^\ast(\bT)\}.$$
Similarly, define
\begin{equation*}
\begin{aligned}
&\ft(E)_{r^+}:=\{H\in \ft(E)\mid \ord(d\chi(H))> r \text{ for all }
\chi \in X^\ast(\bT)\};\\
&\bT(E)_{r^+}:=\{t\in \bT(E)\mid \ord(\chi(t)-1)> r \text{ for all }
\chi \in X^\ast(\bT)\}.
\end{aligned}
\end{equation*}

Once and for all, fix a  splitting
$({\bf B}, \bT, \{x_{\alpha}\})$ of $\bG^\spl$, defined over $\Q$.
This splitting determines a well-defined subgroup $G_0=\bG(\ri_L)$ of $\bG(L)$.
Let $U_\alpha$ be the one-parameter subgroup corresponding to $x_\alpha$:
$U_{\alpha}=1+L x_\alpha$.
Let $\psi=\alpha+n\in \tilde \Phi$ be an affine root.
Define
\begin{equation}\label{eq:Upsi}
U_\psi=\{g \in U_{\alpha}\mid g=1+tx_\alpha, \quad \ord(t)\ge n\}.
\end{equation}
Note that $U_{\alpha+0}=U_\alpha\cap G_0$.
Similarly, one can define the sublattices $\fu(L)_\psi\subset \fg(L)$
(with each $\fu(L)_\psi$ contained in the root subspace $\fg(L)_\alpha$, where $\alpha$ is the gradient of $\psi$).

Finally, let $x\in \scA(\bT, L)$, $r\in \R$.
Then one can define
\begin{equation*}
\begin{aligned}
&\fg(L)_{x,r}=\ft(L)_r\oplus \sum_{\{\psi\in \tilde \Phi\mid \psi(x)\ge r\}}\fu(L)_\psi\\
&\fg(L)_{x,r^+}=\ft(L)_{r^+}\oplus \sum_{\{\psi\in \tilde \Phi\mid \psi(x)> r\}}
\fu(L)_\psi.
\end{aligned}
\end{equation*}
Similarly for the group, for $r\ge 0$,
define $\bG(L)_{x,r}$ as the subgroup of $\bG(L)$ generated by $\bT(L)_r$
and the subgroups $U_\psi$ with $\psi(x)\ge r$,  and
$\bG(L)_{x,r^+}$ as the subgroup of $\bG(L)$ generated by $\bT(L)_{r^+}$
and the subgroups $U_\psi$ with $\psi(x)> r$.

Let $\tilde \R$ be the set $\R\cup \{s^+\mid s\in \R\}$, with the natural ordering
(see e.g. \cite{adler-debacker:mk-theory}*{\S 1.1} for details).

The key fact (quoted in this form from
\cite{adler-debacker:mk-theory}*{Lemma 2.2.1, Remark 2.2.2}) we use is that
since $L/\K$ is a tamely ramified Galois extension,
\begin{enumerate}
\item\label{fix1} $\scB(\bG, L)^{\Gamma}=\scB(\bG, \K)$, and
\item\label{fix2} for $x\in \scB(\bG, \K)$,
$
(\fg(L)_{x,r})^\Gamma = \fg(\K)_{x,r},\quad \text{ for } r\in \tilde \R,
$
\item\label{fix3}
$(\bG(L)_{x,r})^\Gamma = \bG(\K)_{x,r}$ for  $r\in \tilde \R_{>0}$.

\end{enumerate}
Note that if $L/\K$ is unramified, the equality in
(\ref{fix3}) holds for $r=0$ as well.

For a non-split group, we will use (\ref{fix2}) as a definition of the filtration lattices $\fg(\K)_{x,r}$, $r\in \tilde\R$,  and use (\ref{fix3}) as the definition
of the filtration subgroups
$\bG(\K)_{x,r}$, $r\in \tilde \R_{>0}$.

The definition of the parahoric subgroups $\bG(\K)_{x,0}$ for a group that splits over a ramified extension is more complicated, and does not readily translate to Denef-Pas language (which is our main goal in recalling the definitions). We will show below that for our purposes we can replace $\bG(\K)_{x,0}$ with the (in general, larger) set
$(\bG(L)_{x,0})^\Gamma$.

\begin{defn}
Define
$$\fg(\K)_r:=\bigcup_{x\in \scB(\bG, \K)}\fg(\K)_{x, r}, \quad\text{and}
\quad \bG(\K)_r:=\bigcup_{x\in \scB(\bG, \K)}{\bG}(\K)_{x,r}.
$$
\end{defn}
Then the sets $\fg(\K)_r$ and $\bG(\K)_r$ are open and closed, and are both
$\bG(\K)$-domains.

\subsection{Definability}
Here we collect some basic statements about definability (or in one case, almost-definability) in Denef-Pas language of the filtration subgroups (respectively, the corresponding lattices in the Lie algebra) defined above.

\begin{lem}\label{lem:gxr}
Let $x\in {\mathfrak P}_\Psi$ be an optimal point.
Then the sets $\fg(\K)_{x, r}$ and $\fg(\K)_{x, r^+}$ are definable using the parameter
in $z\in Z_{[\Gamma]}$.
\end{lem}

\begin{proof}
Consider  the split case first.
By definition,
$$\fg(\K)_{x,r}=
\ft(\K)_r\oplus \sum_{\{\psi\in \tilde\Phi\mid \psi(x)\ge r\}}\fu(\K)_\psi.$$
Since the set of values $\{\psi(x)\}_{\psi\in \Sigma}$ (where $\Sigma$ is the set of affine roots from the definition of an optimal point) is field-independent by \S \ref{subsub:OP}, the indexing set in the sum is a field-independent set determined by the point $x$;
each set $\fu(\K)_\psi$ is definable by definition, cf.(\ref{eq:Upsi}).
Note that due to natural inclusions between the sets  $\fu(\K)_\psi$ for the affine roots $\psi=\alpha+n$ with the same gradient $\alpha$, the above sum in fact has finitely many non-redundant terms, and the number of these terms is field-independent.

The set $\ft(\K)_r$ is clearly definable. Hence, the sum is definable. The same argument applies to $\fg(\K)_{x,r^+}$.
The non-split case follows from the split case. Indeed, by our definition, $\fg(\K)_{x,r}$ is the set of $\Gamma$-fixed points of $\fg(L)_{x,r}$, which we just proved is definable. The group $\Gamma$ acts by linear transformations (which depend on the parameter in $z\in Z_{[\Gamma]}$); hence, the set of fixed points is definable, using the parameter
$z\in Z_{[\Gamma]}$.
We note that in the split case, a similar lemma was first proved by J. Diwadkar, \cite{diwadkar:thesis}*{Lemma 78}.
\end{proof}

\begin{cor}\label{lem:gr}
\begin{enumerate}
\item Fix $r\in \R$. Then
the sets $\fg(\K)_r$ and $\fg(\K)_{r^+}$
are definable with parameters in $Z_{[\Gamma]}$.
\item If we let $l$ vary, then $\{{\bf 1}_{\fg(\K)_l}\}_{l\in \Z}$ is
a constructible family of motivic functions indexed by $l\in \Z$.
\end{enumerate}
\end{cor}

\begin{proof}
\emph{(1).}
By \cite{adler-debacker:bt-lie}*{Lemma 2.3.2, and Remark 3.2.4}, we have:
$$\fg(\K)_r=\bigcup_{x\in \mathfrak P_{\Phi^\unr, \tau}}{}^{\bG(\K)}\fg(\K)_{x,r}.$$
The finite set of optimal points $\mathfrak P_{\Phi^\unr, \tau}$ depends only on the parameter in $Z_{[\Gamma]}$ (more specifically, there are finitely many possibilities for this set, and the specific choice is determined by the definable subset
$Z_{\Phi^\unr, \tau}\subset Z_{[\Gamma]}$ from \S \ref{sub:building} that contains
the cocycle $z$ defining $\bG$).
Then by the previous lemma, $\fg(\K)_r$ is a finite union (indexed by a field-independent set) of definable subsets, and hence, is definable.
For $\fg(\K)_{r^+}$, there exists an $s\in \R$, such that
$\fg(\K)_{r^+}=\fg(\K)_s$ (cf. \cite{adler-debacker:bt-lie}*{Remark 3.2.4}); hence, the second statement follows from the first.

\emph{(2).} Since the set of optimal points is independent of $l$, we only need to show that for an arbitrary optimal point $x\in \scB(\bG, \K)$,
the set $\fg(\K)_{x, l}$ depends on $l$ in a definable way.
Recall that by our definition,
$$\fg(\K)_{x, l}=\left(\ft(L)_l \oplus
\sum_{\{\psi\in \tilde \Phi\mid \psi(x)\ge l\}} \fu(L)_\psi\right)^\Gamma,$$
where $L$ is the Galois extension that splits $\bG$ and $\Gamma=\gal(L/\K)$, as in \S \ref{sub:roots}.
We see directly from the definitions that both the set $\ft(L)_l$, and the
indexing set $\{\psi \in \tilde \Phi\mid \psi(x)\ge l\}$ are defined by inequalities
with $l$ on one side, and  a definable function  on the other, and thus, depend on $l$ in a definable way.
\end{proof}

There are finitely many conjugacy classes of maximal parahoric subgroups in $\bG(L)$, corresponding to the hyperspecial points of $\scB(\bG, L)$.
One would like to prove that parahoric subgroups corresponding to special points in
$\scB(\bG, \K)$ are definable. Here we prove a weaker statement, sufficient for the purposes of his article. As always, when talking about definability, the residue characteristic of $\K$ is assumed to be sufficiently large.
We note that the split case of the first statement of the following lemma first appeared in \cite{diwadkar:thesis}.

\begin{lem}\label{lem:k0}
  Let $x\in \scB(\bG, \K)$ be a   {special point}.
Let $L$ be a finite tamely ramified Galois extension such that
$\bG$ splits over $L$, and $\Gamma =\gal(L/\K)$, as above. Then
\begin{enumerate}
\item The set $K_0:=\bG(L)_{x}^\Gamma$
is a definable (using a parameter in $z\in Z_{[\Gamma]}$) subset of $\bG(\K)$, and
\item the set $K_0$ contains the parahoric subgroup $\bG(\K)_x$, and
there exists a constant $c$ that depends only on the root datum of $\bG$ such that
$[K_0: \bG(\K)_{x}]\le q^c$, where $q$ is the cardinality of the residue field of $\K$.
If $L/\K$ is unramified, then $K_0=\bG(\K)_x$.
\item For every optimal point $x\in \bar C$ as in \S \ref{subsub:OP} above, for every $r>0$, the subgroup
$\bG(\K)_{x,r}$ is definable.
\end{enumerate}
\end{lem}
\begin{proof} The proof of (1) is almost identical to the proof of
  Lemma \ref{lem:gxr} above. We start with the split case, and examine the definition
of $\bG(L)_x$. This subgroup depends only on the facet that contains $x$; and the values
$\psi(x)$ of the affine roots are   {rational numbers} (independent of the field) determined by the facet.  Thus we have a finite, field-independent set of definable
subgroups $U_\psi$, and a definable subgroup $T(L)_r$. To show that $\bG(L)_x$, which, by definition, is generated by these subgroups, is definable,  it remains to observe that there is a uniform bound on the length of the word of generators required to write down every element. In fact, it follows from Chevalley commutator relations that this length is bounded by $|\Phi|+r$, where $r$ is the absolute rank of the group $\bG$. 
Hence, $\bG(L)_x$ is definable.
Then $\bG(L)_x^\Gamma$ is definable, using the parameter in $z\in Z_{[\Gamma]}$, since $\Gamma$ acts by definable automorphisms.

The statement (2) is Lemma B.14 in \cite{S-T}.

The proof of the statement (3) in the split case (i.e., for $\bG(L)$)
is identical to the proof Part 1 of Lemma \ref{lem:gr} above, with
the lattices $\fu(L)_\psi$ replaced with the subgroups $U_\psi$, and using the same remark about the bound on the length of the word of generators as in part (1).
The statement for $\bG(\K)_{x, r}$ follows immediately from the split case, since we assume that
$r>0$, by the condition (3) in \S \ref{sub:MP}.
\end{proof}

\subsection{Haar measures}\label{sub:HM}
The functions representing the distributions that we study in this paper depend on the choice of the normalizations of the Haar measures on $\bG(\K)$ and $\fg(\K)$.
However, we note that the questions we are interested in, namely, those of local integrability and local boundedness, are not sensitive to scaling by a constant, hence, any normalization of Haar
 measure that makes it a specialization of a motivic measure will work for our purposes.
Let us describe some aspects of motivic measures related to differential forms and our choices for the normalizations.
 
\subsubsection{}\label{defdif} Given a definable subassignment and a definable differential form on it, there is an associated motivic measure, see \cite{cluckers-loeser}*{\S 8} and \cite{cluckers-loeser:mixed}*{{\S 12.3}}. Let us explain how this works, focusing on what we need in this paper, namely uniformity in $\K$ and in families. A definable differential $d$-form $\omega$ on the affine space $\K^N$ with the coordinates $x_1,\ldots,x_N$ is given by a finite sum of terms of the form $f dx_{i_1}\wedge \dots\wedge dx_{i_d}$ for $0<i_1<\ldots< i_d\leq N$ and where the coefficients $f$ are definable functions from $\K^N$ to $\K$. For a $d$-dimensional $\K$-analytic definable submanifold $A\subset \K^N$, such a $d$-form $\omega$ gives a measure on $A$, usually denoted by $|\omega|$, since any definable function is $\K$-analytic away from a definable set of smaller dimension. This construction goes back to Weil and is detailed in
\cite{Bour}. Here we  explain how to think of the measures defined by volume forms  
uniformly in $\K$ and in families. 

By Lemma 11.3 of \cite{cluckers-loeser:mixed}, and by observing the construction in its proof, the following holds for any definable set $A\subset \K^N$ of dimension $d$. There exist an integer $s\geq 0$ and a definable bijection $f:A'\subset \rf_\K^s\times A\to A$, induced by the coordinate projection, such that for each $\xi\in \rf_\K^s$, the fiber $A_\xi':=\{x\mid (\xi,x)\in A'\}$ is a $\K$-analytic manifold of dimension at most $d$ for which there is a definable, isometric isomorphism of $\K$-analytic manifolds $\iota:A'_\xi\to B_\xi\subset \K^d$ which is induced by one of the coordinate projections $\K^N\to \K^d$. Clearly there is a definable differential $d$-form $\omega_\xi$ on $\K^d$ whose restriction to $B_\xi$ coincides with the pullback of the restriction of $\omega$ to $A'_\xi$ under $\iota^{-1}$.  Now $\omega_\xi$ is of the form $fdx_1\wedge\ldots \wedge dx_d$ for some definable function $f$, and a measure can be defined on $\K^d$ (and hence on each of the $B_\xi$ and the  $A'_\xi$), by defining the measure of an open set $U$ in $\K^d$ as the integral of $|f|$ over $U$ against the Haar measure on $\K^d$, normalized so that the unit ball has measure $1$. This gives the measure $|\omega|$ on $A$. This construction of forms and measures works uniformly in families and also when $\K$ varies, and corresponds to the motivic treatment of forms and measures in \cite{cluckers-loeser}*{\S 8} and \cite{cluckers-loeser:mixed}*{{\S 12.3}}.

\subsubsection{}\label{subsub:invforms}
For a split connected reductive group $ {\bG^\spl}$, one can explicitly write down a definable differential form, {which we denote by $\omega^\spl$}, that gives rise to a Haar measure on
$ {\bG^\spl(\K)}$ (see e.g. \cite{cluckers-hales-loeser}*{\S 7.1}).
For a non-split reductive group, we define the invariant differential form on it by pull-back from 
its split form, using the same construction as Gross uses for an inner form, cf.\cite{gross:motive}*{(4.8)}. Since here we are working with not necessarily inner forms, we need to generalize this construction slightly to allow the volume form to have coefficients in an extension of the field $\K$, and in the end we do not generally get the canonical measure of Gross.

As in \S \ref{sub:roots}, we think of a general connected reductive group $\bG$ as a fibre 
$\bG=\bG_z$ of a definable subassignment (constructed in \S \ref{sub:roots})
over a point $z\in Z_{[\Gamma]}$. Recall that the subassignment $Z_{[\Gamma]}$ specializes to
the set of cocycles that give rise to forms of a given split reductive group.
For every such form $\bG_z$ with $z\in Z_{[\Gamma]}$, we have an isomorphism
$\psi_z:L_z\otimes_\K \bG_z \to L_z\otimes_\K \bG^\spl$, where $L_z$ is the finite extension over which $\bG_z$ splits. Being an isomorphism of algebraic groups (defined over $L_z$), the map
$\psi_z$ is definable, using $z$ as a parameter  (recall that as discussed in \S \ref{sub:roots}, using the parameter $z$ allows us to use the elements of $L_z$ in all Denef-Pas formulas).

At the identity, the map $\psi_z$ induces an isomorphism of $1$-dimensional vector spaces (which we denote by the same symbol)
$$\psi_z:\wedge^{\dim \bG}(L_z\otimes_\K \fg_z) \to \wedge^{\dim \bG}(L_z\otimes_\K \fg^\spl),$$
defined over $L_z$ and which is, clearly, definable using the parameter $z$. (We observe that
in the case $\psi_z$ is an inner twisting, this isomorphism is actually defined over $\K$
(cf. \cite{laumon:Drinfeld-modules}, pp.68-69), but we do not need this fact here.)
Similarly, for every $x\in \bG(\K)$,  the map $\psi_z$ induces an isomorphism (over $L_z$) of 
$\wedge^d T_x(\bG)\otimes_\K L_z$ and $\wedge^d T_x(\bG^\spl)\otimes_\K L_z$, where $T_x$ denotes the tangent space at $x$. We still denote this isomorphism by $\psi_z$.

Let $\omega_z=\psi_z^\ast(\omega^\spl)$. Then $\omega_z$ is a non-vanishing (since $\psi_z$ is an isomorphism at every $x\in \bG(\K)$) top-degree differential form on $\bG(\K)$, of the form
$f(x_1, \dots x_d) dx_1\wedge \ldots\wedge dx_d$ in any coordinate chart  on $\bG(\K)$, 
with $f$ a regular $L_z$-valued function on the corresponding coordinate chart, whose coefficients are definable functions of $z$. Finally, $\omega_z$ is $\bG(\K)$-invariant, since it comes from 
a $\bG(L_z)$-invariant form on $\bG_z\otimes L_z$. 
We can define the measure $|\omega|_{L_z}$ on $\bG(\K)$ associated with the form $\omega_z$ by taking the $L_z$-absolute value of the function $f(x_1, \dots, x_d)$  on each coordinate chart as above, as in \S \ref{defdif}, except replacing the $\K$-absolute value with its unique 
extension to $L_z$.

We see that $(\omega_z)_{z\in Z_{[\Gamma]}}$ is a  constructible family of definable differential forms on the fibres $\bG_z$. 
Therefore, the measures $|\omega_z|_{L_z}$ form a family of 
(specializations of) motivic measures. 
These will be the Haar measures we consider on the groups $\bG_z(\K)$.

\begin{rem} We observe that the resulting Haar measure coincides with the canonical measure defined by Gross \cite{gross:motive} in the case when $\bG$ splits over an unramified extension, cf. \cite{cluckers-hales-loeser}*{\S 7.1},  {as well as in the case when $\bG$ is a non-quasi-split inner twist of $\bG^\spl$, by definition in \cite{gross:motive}};
in the  {quasi-split} ramified case this relationship still needs to be better understood, cf.
\cite{S-T}*{\S B.5.2}.
\end{rem}

\subsubsection{Integrability of motivic functions}
In this article, we occasionally integrate motivic (exponential) functions. We do that without further comment only in the situation when such integration amounts to integration of a 
function which is clearly integrable, such as a Schwartz-Bruhat function in the sense of \cite{cluckers-loeser:fourier}*{\S 7.5}, or a product of a Schwartz-Bruhat function and an additive character.
Such functions are known to be integrable in the sense of motivic integration by the results of
\cite{cluckers-loeser:fourier}, and the convergence of the $p$-adic integrals of their specializations is clear.
In \cite{TI} we prove much stronger results about integrability for motivic functions, allowing us to handle questions of integrability of their specializations
in the sense of $L^1$; here every time there is any issue with the convergence of the integral, we include a careful discussion, and invoke the corresponding transfer principles from \cite{TI}.

\section{Orbital integrals as ``motivic distributions''}

\subsection{The linear dual of $\fg$, and assumptions on $p$}\label{sub:hypo}

Let $\fg(\K)^\ast$ denote the linear dual of $\fg(\K)$.
In \cite{moy-prasad:k-types}*{\S 3.5}, Moy and Prasad define a filtration of
$\fg(\K)^\ast$ by lattices $\fg(\K)_{x,r}^\ast$, where $x$ is a point in the building
$\scB(\bG, \K)$, and $r$ is a real number, by
$$\fg(\K)_{x,r}^\ast=\{\lambda \in \fg(\K)^\ast \mid
\lambda(\fg(\K)_{x, (-r)^+})\subset \mathfrak p_\K\}.$$

{We will need  a particularly nice non-degenerate bilinear form on
$\fg(\K)$; its existence is guaranteed by \cite{adler-roche:intertwining}*{Proposition 4.1}.
We quote this proposition here omitting the details about the list of bad primes. }
\begin{prop}(\cite{adler-roche:intertwining}*{Proposition 4.1})\label{subsub:hy2}
If the characteristic of $\K$ is outside a certain finite list of primes  determined by the root datum of $\bG$, then
there exists a {$\K$-valued}, non-degenerate, bilinear, $\bG(\K)$-invariant, symmetric form
$\langle, \rangle$ on $\fg(\K)$ such that, under the associated identification of $\fg(\K)$
with $\fg(\K)^\ast$, for all $x\in \scB(\bG, \K)$ and all $r\in \R$,   {the lattice}
$\fg(\K)_{x, r}$  {is identified with} $\fg(\K)_{x,r}^\ast$.
\end{prop}

Note that the proof of \cite{adler-roche:intertwining}*{Proposition 4.1} is, in fact, constructive, so we can use the form constructed in the proof from now on to
identify $\fg(\K)^\ast$ with $\fg(\K)$.
This allows us to think of $\fg(\K)^\ast$ as a definable set, identical to $\fg(\K)$.
Recall that in the beginning, we have fixed an embedding of $\fg^\spl$ into $\gl_n$, defined over $\Q$; that leads to a consistent choice of coordinates on all its forms $\fg=\fg_z$, so let us, once and for all, fix the coordinates on the linear space $\fg(\K)$.
Let $X=(x_i)_{i=1}^d$ with respect to these coordinates, where $d$ stands for $\dim \bG$.
Examining the proof of \cite{adler-roche:intertwining}*{Proposition 4.1},
we observe that for $X, Y\in \fg(\K)$, the value
$\langle X, Y \rangle$ is a definable function of the coordinates
of $X$ and $Y$ (using $z$ as a parameter).
Thus, to us, $\fg(\K)^\ast$ is the same definable set as $\fg(\K)$, with the same Haar measure,  though we
keep the $\ast$ for the convenience of interpretation,  and sometimes denote this set
by $\fg^\ast(\K)$.

\subsubsection{Nilpotent orbits}
In the rest of this paper, we will need to make
assumptions on the characteristic of the field, which guarantee that the nilpotent orbital integrals are sufficiently well-behaved.

For an element of a Lie algebra $\fg(\K)$,
there are several definitions of ``nilpotent'' in the literature; it turns out that for large $p$ (or in characteristic zero) they are all equivalent, but we will not use this fact here.
We adopt the same definition as in \cite{debacker:nilp}, namely,
we call an element $X\in \fg(\K)$ \emph{nilpotent} if there exists
$\lambda\in  X_{\ast}^\K(\bG)$ such that
$$\lim_{t\to 0}\Ad(\lambda(t)) X = 0.$$
Following DeBacker, we denote by $\cO(0, \K)$ the set of orbits of nilpotent
elements.

\subsubsection{The assumptions on $p$}
Everywhere from now on, we need to assume that the characteristic of the field
is sufficiently large so that all of the following conditions hold:
\begin{enumerate}
\item there are finitely many nilpotent orbits in $\fg$;
\item the nilpotent orbital integrals are distributions on $\fg(\K)$;
\item the bilinear form from Proposition \ref{subsub:hy2} exists.
\end{enumerate}
It is proved in \cite{mcninch:nilpotent} and also \cite{debacker:homogeneity} that
 when $p$ is larger than some constant that can be computed from the absolute
 root  datum of $\bG$, the first two conditions hold.
 Let $M_\Psi^\nilp$ denote the constant such that for
 $p\ge M_\Psi^\nilp$
the above conditions hold.
We also enlarge $M_\Psi^\nilp$ if necessary, and assume that $M_\Psi^\nilp \ge M_\Psi$,
where $M_\psi$ is the constant of \S \ref{sub:roots}.

\subsection{Orbital integrals as motivic distributions}\label{sub:noi}
Let $X\in \fg(\K)$. The definition of an orbital integral $\Phi_X$, which is a distribution on the space $C_c^\infty(\fg(\K))$, is recalled in \S \ref{app:oi}.
In this paper, rather than use the approach to the definition of an orbital integral that requires us to fix  a Haar measure on $\bG(\K)$ and a Haar measure on the centralizer $C_G(X)$, we will
use  a specific $\bG(\K)$-invariant differential form on the orbit of $X$. We start by recalling the construction of such a form, that goes back to Kirillov. The version we use here is quoted from \cite{kottwitz:clay}*{\S 17.3}.
\subsubsection{Invariant volume forms on orbits}
Let $\langle, \rangle$ be the (definable) non-degenerate, symmetric, $\bG(\K)$-invariant
bilinear form on $\fg(\K)$ from Proposition \ref{subsub:hy2}.
We use this form to identify $\fg(\K)$ with its linear dual $\fg^\ast(\K)$, as above; this also identifies adjoint orbits with co-adjoint orbits, since the form $\langle, \rangle$ is $\bG(\K)$-invariant.

Let $X\in\fg(\K)$, and let $\cO_X$ be the adjoint orbit of $X$.
The orbit $\cO_X$, as a $p$-adic manifold, is identified with $\bG(\K)/C_G(X)$, and its tangent space at $X$ is identified with $\fg/\fg_X$, where, with our identification of $\fg$ and
$\fg^\ast$,
\begin{equation}\label{gx}
\fg_X=\{Y\in \fg\mid \langle X, [Y,Z]\rangle=0 \text{ for all } Z\in \fg\}.
\end{equation}
Consider the alternating form
\begin{equation}\label{eq:w}
\omega_X(Y, Z)=\langle X, [Y, Z]\rangle
\end{equation}
on  $\fg$.
This form descends to a non-degenerate alternating form on $\fg/\fg_X$, \cite{kottwitz:clay}*{\S 17.3}, and therefore, it gives a symplectic form on the tangent space to $\cO_X$ at $X$.
Let $\omega$ be the symplectic form on $\cO_X$  defined by $\omega(X)=\omega_X$.
Then $\omega$ is a non-degenerate symplectic form on $\cO_X$.
Note that the form $\omega$ depends only on the choice of the bilinear form $\langle, \rangle$.
It follows that the orbit $\cO_X$ is a symplectic manifold, and in particular, its dimension is even. Let $\dim \cO_X=2m$.
Consider the form $\nu_m=\wedge^m \omega$ -- it is a non-vanishing $\bG(\K)$-invariant top degree differential form on
$\cO_X$, that is, a volume form, see \cite{kottwitz:clay}*{\S 17.3}.

\begin{defn}
In this paper, we denote by $\Phi_X$ the  orbital integral (as a distribution)
over the orbit of $X$ equipped with the measure $|d\nu_m|$, where $\nu_m$ is the volume form on
$\cO_X$ as above (we assume that the bilinear from of Proposition \ref{subsub:hy2} is fixed once and for all), and $m=\frac12\dim\cO_X$.
\end{defn}



The next proposition states that for every integer $m$, $1\le m \le (\dim \bG-\operatorname{rank}\bG)/2$, the volume forms
$\nu_m$ defined above form a family of definable volume forms on orbits of dimension $2m$. As a corollary, we obtain that
the orbital  integrals $\Phi_X$ form a family of ``motivic distributions'', in the sense that the result of the application of $\Phi_X$ to any constructible family of definable test functions is a motivic  function in $X$ and the parameters indexing the family. We note that a similar statement (just for semisimple elements, and with less detail) appears also in \cite{S-T}*{\S B.5.3}.

\begin{prop}\label{prop:oii}
Fix an integer $m$, $1\le m \le (\dim \bG-\operatorname{rank} \bG)/2$. Then for the fields $\K$ of sufficiently large residue characteristic, the set of elements $X\in \fg(\K)$ such that $\dim \cO_X=2m$ is definable, and
by restricting the form $\nu_m$ to the orbits $\cO_X$ of dimension $2m$, we obtain a family  of definable
 {$\bG(\K)$-invariant volume} forms on these orbits.
\end{prop}
\begin{proof}
First, observe that the vector spaces $\fg_X$ of (\ref{gx}) form a family of definable sets indexed by $X$, as is clear from their definitions, since the form $\langle, \rangle$ is definable, as discussed below Proposition \ref{subsub:hy2}.
Now, for every integer $k$ between the rank and the dimension of $\bG$ of the same parity as $\dim \bG$, let
Let $V_k=\{X\in \fg(\K)\mid \dim \fg_X=k \}$. The  sets $V_k$ partition $\fg(\K)$, and they are definable since
for every positive integer $k$, the statement $\dim \fg_X=k$ can be expressed by means of Denef-Pas language formulas stating that there exists a collection of $k$ vectors forming a basis of the linear space $\fg_X$
(cf. \cite{gordon-hales:transfer}).

Let $\omega$ be the differential $2$-form on $\fg(\K)$, which, at every $X$, coincides with
the alternating form $\omega_X$ of (\ref{eq:w}), viewed as an element of $\wedge^2 \fg(\K)^*$.
Recall that we have chosen the coordinates $(x_i)_{i=1}^d$ on $\fg(K)$ (see the discussion after
Proposition \ref{subsub:hy2}).
We observe that the coefficients of the form $\omega$  are
linear functions of $x_i$, with coefficients that are definable $\K$-valued functions of the parameter $z$, since the form $\langle, \rangle$ is definable using  $z$ as a parameter.
For every integer  $m\le (\dim\bG -\operatorname{rank}\bG)/2$, consider the $2m$-form $\nu_m:=\wedge^m\omega$ on $\fg$.
Then it is a definable $2m$-form on the definable subassignment $\fg$, and it has the form
$\nu_m=\sum_S f_S \wedge_{i\in S} dx_i$, where $S$ runs over the subsets of cardinality $2m$ of $\{1, \dots, n\}$, and $f_S$ are polynomials with coefficients that are definable functions of $z$. 

By restricting the definable form $\nu_m$ from $\fg$ to the definable family of orbits of elements of $V_{d-2m}$, we obtain a family (indexed by $X$ and $z$) of definable $2m$-forms on the orbits $\cO_X$ with $X\in V_{d-2m}$.
By \cite{kottwitz:clay}*{\S 17.3} (as summarized above), the specialization of the definable
form $\nu_d$ to $\K$ coincides with a $\bG(\K)$-invariant volume form on the orbit of $X$.
Thus, we obtain that the measures on the orbits that we have defined above come  from a family of definable volume forms.
\end{proof}

\begin{cor}\label{prop:noi}
Given a family of motivic exponential test functions
$\{f_a\}_{a\in S}$, such that $f_{a, \K}\in C_c^\infty(\bG(\K))$,
where $S\in \de$ is some definable subassignment,
there exists a constant $M>0$, and a motivic exponential function $F$ on
$\fg\times S$, such that for all local fields $\K\in \cA_M\cup \cB_M$, for every
$X\in \fg_{\K}$,
we have 
$$
\Phi_{X}(f_{a, \K})= F_\K(X, a) \quad \text{for all } a\in S_\K.
$$
(Note that naturally, both $F$ and $M$ depend on the family $\{f_a\}_{a\in S}$).
\end{cor}
\begin{proof}
The statement follows immediately from the above proposition, the constructions of \S\ref{defdif} and \S \ref{subsub:invforms} and the main theorems on motivic integration of Section 4 in \cite{cluckers-loeser:fourier}.  
\end{proof}


\section{Local integrability in large positive characteristic}

\subsection{The function $\eta$}\label{sub:eta}
In \cite{adler-debacker:mk-theory}*{Appendix A}, R. Huntsinger proved an
integral formula for the function representing the Fourier transform of an invariant distribution, which plays a crucial role in our poof.
We need to quote some definitions from
\cite{adler-debacker:mk-theory}*{Appendix A}.
Let $\Lambda$ be an additive character of $\K$ with conductor $\mathfrak p_\K$.
Here we will make use of the notation defined in Appendix A, see \S \ref{sub:inv.distr} below; we will also use the Fourier transform on $\fg$; the definition is recalled  in \S \ref{sub:FT}.

\begin{defn}\cite{adler-debacker:mk-theory}*{Definition A.1.1}\label{def:eta}
Let $K$ be an open compact subgroup of $\bG(\K)$.
For $\lambda\in \fg^\ast(\K)$
and $X\in \fg(\K)$, define
$$
\eta_X(\lambda)=\int_K\Lambda\left(\lambda(\Ad(k)X)\right)\, dk,
$$
where $dk$ is the Haar measure on $K$ normalized so that the volume of $K$
is $1$.
\end{defn}
In this definition, the subgroup $K$ is arbitrary.
Later we will need to assume that it is definable.
Once and for all, let us pick a definable open compact subgroup, for example, take
$K=\bG(\K)_{x_0,r}$ for some fixed optimal point $x_0$ in an alcove $\bar C$ as in \S\ref{sub:MP} (i.e., more precisely,  we pick an arbitrary tuple of baricentric coordinates of $x_0$, that is, an element of the set $\mathfrak P_{\Phi^\unr, \tau}$ of \S \ref{subsub:OP}),
and an arbitrary $r>0$, say, $r=1$ (such a subgroup is definable by Lemma \ref{lem:k0}). Everywhere below, we will use this subgroup $K$.
Let $c_\K$ be the volume of this subgroup $K=\bG(\K)_{x_0, 1}\subset \bG(\K)$ with respect to the measure on $\bG(\K)$ defined in \S \ref{sub:HM}.
Then $c_\K$ is a ``motivic constant'', that is, a motivic function on a point.
Since $c_\K$ might not be invertible in the ring of motivic functions, we cannot say that the motivic measure will be normalized to give $K$ volume $1$; instead, the constant $c_\K$
will appear as a denominator every time we need to replace the integral over $K$ with a motivic integral.

\begin{defn}
Fix $r\in \R$. Let ${\bf 1}_{\fg^\ast(\K)_r}$ denote the characteristic function
of the set $\fg^\ast(\K)_r\subset \fg^\ast(\K)$.
Let
$$\eta_{X, r}:=\eta_X{\bf 1}_{\fg^\ast(\K)_r}.$$
\end{defn}

For every $r\in \R$, the function $\eta_{X, r}$
belongs to the space $\tf(\fg^\ast(\K)_r)$,  see
\cite{adler-debacker:mk-theory}*{Corollary A.3.4}.
Then, given a distribution $\mu\in J(\fg^\ast(\K)_r)$ (where we use the notation recalled in 
\S \ref{sub:inv.distr} below), it makes sense to write
$\mu(\eta_X)=\mu(\eta_{X, r})$.
Moreover, the map $X\mapsto \eta_{X, r}$ is locally constant in $X$.
The main theorem of \cite{adler-debacker:mk-theory}*{Appendix A} is:
\begin{thm}\cite{adler-debacker:mk-theory}*{Theorem A.1.2}\label{thm:eta}
Fix $r\in \R$. Let $\mu\in J(\fg^\ast(\K)_r)$. Then $\widehat \mu$ is represented on $\fg^\reg$ by the locally constant function
$X\mapsto \mu(\eta_X)$.
\end{thm}

Using the bilinear form $\langle,\rangle$ to identify $\fg(\K)$ with
$\fg^{\ast}(\K)$, we can transport the function $\eta$ to $\fg(\K)$.
\begin{defn}\label{def:tildeta}
Let $\tilde \eta$ be the function on $\fg(\K)$ defined by:
$$
\tilde \eta_X(Y)=\int_K\Lambda(\langle \Ad(k)X, Y \rangle )\, dk,
$$
and let
$$\tilde \eta_{X, r}:=\tilde \eta_X{\bf 1}_{\fg(\K)_r}.$$
\end{defn}

Now, observe that if we identify $\fg(\K)$ with $\fg^\ast(\K)$ using
the form $\langle, \rangle$, then the space of distributions
$J(\fg^{\ast}(\K))$ is identified with the space $J(\fg(\K))$. Since
the set $\fg^\ast(\K)_r$ is identified with $\fg(\K)_r$ for all $r\in \R$,
the space $J(\fg^\ast(\K)_r)$ is identified with $J(\fg(\K)_r)$.
Now Theorem \ref{thm:eta} can be restated as:
\begin{cor}\label{cor:eta1}
Fix $r\in \R$. Let $\mu\in J(\fg(\K)_r)$.
Then $\widehat \mu$ is represented on $\fg^\reg$ by the locally constant
function
$X\mapsto \mu(\tilde\eta_X)$.
\end{cor}

\subsection{Fourier transforms of orbital integrals}
In this section, we {use the bilinear form of Proposition \ref{subsub:hy2}}.
We continue to work with the function $\tilde\eta_X$ from Definition \ref{def:tildeta} above.

\begin{lem}\label{lem:eta}
Up to a constant, the functions $Y\mapsto \tilde\eta_X(Y)$ form a constructible family of
  motivic exponential functions (indexed by $X\in \fg$).
More precisely, there exists a   motivic exponential function
$\Upsilon$ on $\fg\times \fg$, and a constant $M_\Upsilon$ such that for
all $\K\in \cA_{M_\Upsilon}\cup \cB_{M_\Upsilon}$, we have
$$\frac{1}{c_\K}\Upsilon_\K(X, Y)=\tilde\eta_X(Y),
$$
for all $(X, Y)\in \fg(\K)\times \fg(\K)$, where $c_\K$ is the motivic constant that is defined
below Definition \ref{def:eta}.
\end{lem}

\begin{proof}
In the definition of the function $\tilde \eta_X$, the compact open
subgroup $K$ is arbitrary. Pick the definable subgroup $K$ discussed after the Definition \ref{def:eta}.
Then the statement follows immediately from the main theorem about motivic integrals of   motivic exponential functions,
\cite{cluckers-loeser:fourier}*{Theorem 4.1.1}, which states that if we integrate a   motivic exponential function with respect to some of its variables, with respect to a motivic measure, the result is a   motivic exponential function of the remaining variables.
Note that here the integral over $K$ is with respect to the measure that comes from the differential form on $\bG(\K)$ defined in \S \ref{sub:HM}, which requires a slight generalization of the framework of motivic integration, which is described in \S \ref{subsub:gen}.
\end{proof}

Recall that the locally constant function on $\fg^\reg$ representing the  Fourier
transform of $\Phi_X$ is denoted by $\widehat\mu_X$, (cf. Appendix A).
In order to prove the next theorem about local integrability of the functions $\widehat\mu_X$ in positive characteristic, we need a family version
of Lemma \ref{lem:eta} for the  functions $\tilde\eta_{X,r}$ as $r$ varies.
More precisely, it will be sufficient to consider the
family of functions $\tilde\eta_{X, l}$, as $l$ runs through the integers.
{Here we use the subscript `$u$' in all the notation to emphasize that this is a \emph{uniform in a family} version of the corresponding earlier constructions.}

\begin{lem}\label{lem:eta-uni}
Up to a constant, the functions $Y\mapsto \tilde\eta_{X,l}(Y)$ form a constructible family of
  motivic exponential functions
(indexed by $X\in \fg$ and $l\in \Z$).
More precisely, there exists a   motivic exponential function
$\Upsilon^u$ on $\fg\times h[0,0,1]\times \fg$, and a constant
$M^u_\Upsilon$ such that for
all $\K\in \cA_{M^u_\Upsilon}\cup \cB_{M^u_\Upsilon}$, we have
$$\frac{1}{c_\K}\Upsilon^u_\K(X,l, Y)=\tilde\eta_{X, l}(Y),
$$
for all $(X,l, Y)\in \fg(\K)\times \Z \times \fg(\K)$.
\end{lem}
\begin{proof}
By definition, $\tilde\eta_{X, l}=\tilde\eta_X {\bf 1}_{\fg(\K)_l}$.
By Lemma \ref{lem:eta}, the family $\{\tilde\eta_X\}_{X\in \fg}$
is a constructible
family of motivic exponential functions, and by Corollary \ref{lem:gr}, Part 2, the functions
${\bf 1}_{\fg(\K)_l}$ also form a constructible family indexed by $l\in \Z$;
thus,  $\tilde\eta_{X, l}$ is a constructible family of motivic exponential functions, indexed by $\fg\times \Z$.
\end{proof}

Now we are ready to prove the main theorem of this section.
\begin{thm}\label{thm:ssint}
There exists a constant $M$ depending only on the root datum of $\bG$,
and a motivic exponential function $H$ on $\fg\times \fg^\reg$
such that for every local field
$\K\in \cA_{M}\cup \cB_{M}$,
for every $X\in \fg$, we have
$$\widehat\mu_X(Y)=\frac{1}{c_\K}H_\K(X, Y)$$
for all $Y\in \fg^\reg_\K$,
where $c_\K$ is the volume of the subgroup $K$, see the discussion after Definition \ref{def:eta}.
\end{thm}

\begin{proof}
We recall that $\fg(\K)=\bigcup_{r<0}\fg(\K)_r$.

For the moment, fix a field $\K$; let $X\in \fg(\K)$, and let $r\in \R$ be an arbitrary real number such that
$X\in \fg(\K)_r$.
Then the distribution $\Phi_X$ lies in the space  $J(\fg(\K)_r)$, and
by Huntsinger's formula (of Corollary \ref{cor:eta1}), we have
$\widehat\mu_X(Y)=\Phi_X(\tilde \eta_{Y,r})$, where $\tilde\eta_{Y,r}$ are the functions from \S \ref{sub:eta}. Note that the right-hand side does not depend on $r$ as long as $X\in \fg(\K)_r$.
Thus, for every integer $l$, we  have
$$\widehat\mu_X(Y)=\Phi_X(\tilde \eta_{Y,l}) \text{ for } X\in \fg(\K)_l,
\ Y\in \fg(\K)^\reg. $$
By Lemma \ref{lem:eta-uni}, for $\K\in \cA_{M_\Upsilon^u}\cup \cB_{M_\Upsilon^u}$, we have
that $c_\K\tilde \eta_{Y, l}$ is the specialization to $\K$ of the
motivic exponential function $\Upsilon^u$.
Take the family $\tilde\eta_{Y, l}$ (indexed by $Y\in \fg^\reg$ and $l\in \Z$)
as the family of test functions. The theorem now follows by applying
Proposition \ref{prop:noi} to this family.
\end{proof}

\subsubsection{Proof of Theorem \ref{thm:orb int loc int}}\label{subsub:th1p2}
The statement now follows from Theorem \ref{thm:ssint} and Harish-Chandra's theorems
\cite{hc:queens} (reproduced as Statement (\ref{HC:orb-int}) in \S \ref{sub:HCtheorems}), by the Transfer of local integrability and Transfer of boundedness
principles, \cite{TI}, quoted here as Theorems \ref{thm:ti} and \ref{thm:bounded}. 
Indeed, Harish-Chandra's result asserts nice-ness of $\widehat \mu_X$ when the characteristic of $\K$ is zero, and the transfer results of \cite{TI}, which apply thanks to Theorem \ref{thm:ssint},  yield that the conditions of nice-ness (local integrability  and local boundedness) are independent of the characteristic of $\K$, once the residue field characteristic is sufficiently large.
\qed

Let us summarize the origins of the constant $M_\bG^\orb$ that provides the restriction
on the characteristic in this theorem.
\begin{enumerate}
\item $M_\bG^\orb  \ge M_\Psi$ (where $M_\Psi$ is the constant defined in
\S \ref{sub:roots}), so that the group $\bG(\K)$ indeed appears as an element in the constructible family
for some parameter in $Z_{[\Gamma]}$ with a suitable $\Gamma$.

\item We assume that $M_\bG^\orb \ge M_\Psi^\nilp$, which ensures that it  is large enough so that  there are finitely many nilpotent orbits, and nilpotent orbital integrals are well-defined distributions.
It is also large enough so that the bilinear form
$\langle, \rangle$ of \S \ref{subsub:hy2} exists.

\item $M_\bG^\orb  \ge M_{\Upsilon}^u$, where $M_\Upsilon^u$ is the constant defined in
Lemma \ref{lem:eta-uni}, so that the family of motivic functions of Lemma \ref{lem:eta-uni} specializes to Huntsinger's functions $\tilde \eta_{X, l}$.

\item $M_\bG^\orb $ needs to be large enough so that for the family
of functions $\tilde\eta_{X, l}$, the motivic integrals
over the orbits specialize to the orbital integrals, see Proposition \ref{prop:noi}.

\item Finally, $M_\bG^\orb $ might need to be enlarged further
so that transfer of integrability holds for the motivic exponential function $H(X, Y)$ that specializes to $c_\K\widehat \mu_X(Y)$.
\end{enumerate}

\subsection{Harish-Chandra characters}\label{sec:char}
Let $M_\bG$ be a constant, determined by the root datum of $\bG$, such that all the hypotheses listed in  \S\ref{sub:loc_gamma} hold. That is, we take $M_\bG$ to be the maximum of the constants
$M_{\bM}^\orb $ of Theorem \ref{thm:orb int loc int} for every $\bM$ on the list of possible reductive groups that can arise as the connected component of a centralizer of a semisimple element of $\bG$.

\subsubsection{Proof of Theorem \ref{thm:char}}\label{subsub:th2}
Let $\K\in \cB_{M_\bG}$, where $M_\bG$ is defined above, and let $\pi$ be an admissible representation of $\bG(\K)$.
First, we prove this theorem in a neighbourhood of the identity, more precisely, on the set
$\bG(\K)_r$, where $r$ is chosen so that $\fg_r=\fg_{\rho(\pi)^+}$, where
$\rho(\pi)$ is the depth of $\pi$, {or any larger number such that Hypothesis \ref{subsub:exp} holds for $r$}.
The niceness of $\theta_\pi$ restricted to this neighbourhood is immediate; indeed,
by Theorem \ref{thm:expansion}, on
$\bG(\K)_r$ the function $\theta_\pi$ is a finite linear combination of the functions $\widehat\mu_\cO$, and these functions are nice by Theorem \ref{thm:orb int loc int}, Part (1).
 The next statement is an easy technical point: since $\theta_\pi$ is nice,
the integral $\int_{\bG(\K)}\theta_\pi(g) f(g) \, dg$ converges for all test functions
$f$ with support contained in $\bG(\K)_r$, not just those with support contained in the set of regular elements. One still needs to show that this integral coincides with the value of the distribution $\Theta_\pi(f)$ for such functions. This is almost a tautology, based on careful reading of the work of DeBacker.
Indeed, even though the local character expansion is stated in \cite{debacker:homogeneity} as an equality of
functions defined on the regular set, in fact it is proved at the level of distributions, without the assumption that the support of the test function is
contained in the regular set; see the proof of \cite{debacker:homogeneity}*{Theorem 3.5.2}, where (using our notation for the orbital integral)
it is shown that for any
$f\in  C_c^\infty(\fg(\K)_r)$,
$$\Theta_\pi(f\circ \mexp^{-1})=\sum_{\cO\in \cO(0, \K)}c_\cO(\pi)\widehat\Phi_\cO(f).$$
 {Then we have}, for all $f\in C_c^\infty(\bG(\K)_r)$:
\begin{equation*}
\begin{aligned}
\Theta_\pi(f)&=\sum_{\cO\in \cO(0, \K)}c_\cO(\pi)\widehat\Phi_\cO(f\circ\mexp)
=\sum_{\cO\in \cO(0, \K)}c_\cO(\pi)\int_{\fg(\K)}(f\circ\mexp)(X)\widehat\mu_\cO(X)\, dX\\
&=\int_{\bG(\K)}\theta_{\pi}(g)f(g)\, dg,
\end{aligned}
\end{equation*}
 {where now we know that all the integrals converge, by Theorem
\ref{thm:orb int loc int}.}

  Now let us prove that $\theta_\pi$ is nice away from the identity as well.
Our strategy, roughly speaking, is to prove that $\theta_\pi$ is nice in a neighbourhood of every semisimple element, and that any compact set in $\bG(\K)$ can be covered with finitely many such neighbourhoods.
Harish-Chandra's descent \cite{hc:van-dijk}*{Chapter 6}
allows to reduce the statement about $\theta_\pi$ in a  {$\bG(\K)$}-neighbourhood
 of a semisimple element $\gamma$ to a statement about a related distribution $\theta$
defined on a neighbourhood of $\gamma$ inside  the centralizer $M=C_G(\gamma)$ of
$\gamma$.
Finally, on a suitable neighbourhood in $M$, niceness of $\theta$ follows from the local character expansion due to Adler and Korman \cite{adler-korman:loc-char-exp} and the
fact that Fourier transforms of nilpotent orbital integrals are nice, as we have shown above.

  We proceed with the proof. {By our choice of the constant $M_\bG$,}
  all the hypotheses listed in \S \ref{sub:loc_gamma} hold.
  {Let $\gamma\in \bG(\K)$ be an arbitrary semisimple element, let $\bM$ be the algebraic group such that $C_G(\gamma)=\bM(\K)$, let $M=\bM(\K)=C_G(\gamma)$, and let all the remaining notation be as in
\S \ref{sub:loc_gamma}. Let $r>\max\{\rho(\pi), 2s(\gamma)\}$.}

   { Let $\theta$ be the distribution on $M$ defined in \cite{rodier:loc-int}*{Proposition 1} (cf. also \cite{adler-korman:loc-char-exp}*{\S 7}, where the same definition is explained for the restriction of $\theta$ to $M_r$)}.
By  \cite{hc:van-dijk}*{Corollary from Theorem 11, p.49}, if we show that $\theta$
is represented by a nice function (which is also denoted by $\theta$, by slight abuse of notation) on the set $M$, it will follow that the function $\theta_\pi$ is nice on
$\bG(\K)$. So, it remains to prove that $\theta$ is a nice function on $M$.

By Theorem \ref{thm:expansion_ss} {(and using the notation of that theorem)}, for
$Y\in \fm_r''$, the   {function $\theta(\gamma \mexp(Y))$} is a finite linear combination of the functions $\widehat \mu_{\cO}(Y)$, as $\cO$ runs over the set of nilpotent orbits in $\fm$.
Let us extend both sides by zero to $\gamma M_r$. Then we have the equality on
$\gamma M_r$.
By Theorem \ref{thm:orb int loc int}, $\widehat \mu_{\cO}(Y)$ are nice functions on
$\fm(\K)$.
Hence,   {$\theta$} is a nice function as a function on $\gamma M_r$.
 Now we would like to show that $\theta$ is a nice function on $M$. For this, since it is conjugation-invariant,  it suffices to show that the sets
${}^M\gamma M_r$ (which, by definition, are open in $M$) cover $M$,  {as $\gamma$ runs over  {the set of semisimple elements in} $M$}.

We observe that all semisimple elements of $M$ are covered automatically. Now suppose $m$
is an arbitrary element of $M$. Our assumptions on the characteristic of $\K$ guarantee that $\bM(\K)$ contains semisimple and unipotent parts of its elements.
Then we have  $m=\gamma_s\gamma_u$, with $\gamma_u \in C_{M}(\gamma_s)$.
 Then we can conjugate $m$ by an element of $C_M(\gamma_s)$ so that
 $\gamma_u$ gets replaced by a conjugate that is as close to the identity as we wish; in particular, we can ensure that it is in $M_r$, which completes this part of the proof.

 Finally, an argument identical to that shown above for a neighbourhood of the identity shows that the equality
\begin{equation*}\label{item:conv}
\Theta_\pi(f)=\int_{\bG(\K)}\theta_\pi(g) f(g) \, dg
\end{equation*}
holds for \emph{all} test functions
$f\in C_c^\infty(\bG(\K))$.
\qed

\subsection{General invariant distributions near the origin}

Combining our Theorem \ref{thm:orb int loc int} with DeBacker's results
summarised in Appendix A as Theorem \ref{thm:distr}, we obtain a partial extension  to the large positive characteristic case (in a neighbourhood of the origin)
of Harish-Chandra's theorem about invariant distributions with support bounded modulo conjugation.

\begin{thm}\label{thm:general}
Let $\bG$ be a connected unramified reductive group with the Lie algebra
$\fg$.
Let $M_\bG^\orb $ be the constant from Theorem \ref{thm:orb int loc int}.
Let $\K\in \cB_{M_\bG^\orb }$, and let $T$ be an invariant distribution
on $\fg(\K)$ with support bounded modulo conjugation.
Suppose that
the support of $T$ is contained in $\fg(\K)_{(-r)^+}$ with some $r\in \R$.
 {Then the restriction of $\hat T$ to $\fg(\K)_r$
is represented by a nice function $\vartheta_T$
on $\fg(\K)_r$}.
\end{thm}

\begin{proof} Let $\K\in \cB_{M_\bG^\orb }$, and let $T$ be an invariant distribution on $\fg(\K)$ with support bounded modulo conjugation.
Then the support of $T$ is contained in $\fg(\K)_{(-r)^+}$
for sufficiently large $r>0$. Fix such an $r$, and  let $f$ be an arbitrary test function with support contained in
$\fg(\K)_{r}$. Then by Remark \ref{rem:FTsupport},
the function $\hat f$ belongs to the space $\cD_{(-r)^+}$.
By Theorem \ref{thm:distr},  the restriction of
$T$ to the space $\cD_{(-r)^+}$ is a linear combination of
the nilpotent orbital integrals.
Since by definition, $\widehat T(f)=T(\hat f)$,
this implies that for all functions $f\in \tf(\fg(\K)_r)$,
$$\widehat T(f)=T(\hat f)=\sum_{\cO\in \cO(0, \K)} c_\cO\Phi_\cO(\hat f)=
\sum_{\cO\in \cO(0, \K)} c_\cO\widehat \Phi_\cO(f),$$
{with some constants $c_{\cO}$.}
Therefore on $\fg(\K)_r$, the distribution $\widehat T$ is represented by the function
$\vartheta_T=\sum_\cO c_\cO\widehat \mu_\cO$, which {is nice}, by Theorem \ref{thm:orb int loc int}.
\end{proof}


\appendix

\section{Invariant distributions: classical results}
In this section we review the notation, definitions, and some of the classical results of
harmonic analysis on $p$-adic groups that are relevant to the present paper.

\subsection{Definitions}
As before, $\K$ is a non-Archimedean local field (with no assumption on its characteristic), $\bG$ is a reductive algebraic group over $\K$, and $\fg$ is its Lie algebra.
\subsubsection{Characters}
Let $(\pi, V)$ be an irreducible admissible representation of $\bG(\K)$.
Then the \emph{distribution character} of $\pi$ is the distribution on the
space $\tf(\bG(\K))$ of locally constant, compactly supported functions on
$\bG(\K)$ defined  by
$$\Theta_{\pi}(f)=\tr\int_{\bG(\K)}f(g)\pi(g) \, dg$$ (since $\pi$ is
admissible, the linear
operator on the right-hand side is of finite rank, and hence its trace is well defined).

It was proved by Harish-Chandra in characteristic zero, and in positive characteristic, by G. Prasad \cite{adler-debacker:bt-lie}*{Appendix B}
for connected groups, and by J. Adler and J. Korman in general \cite{adler-korman:loc-char-exp}*{Appendix}
that there exists  a locally constant function $\theta_{\pi}$ defined on the set of regular elements $\bG(\K)^\reg$ that represents the distribution character:

\begin{equation}\label{eq:rep.reg}
\Theta_{\pi}(f)=\int_{\bG(\K)}\theta_{\pi}(g)f(g)\, dg,
\end{equation}
for all $f\in \tf(\bG(\K)^\reg)$.
The function $\theta_\pi$ is called the Harish-Chandra character.


\subsubsection{Orbital integrals}\label{app:oi}
Let $X\in \fg(\K)$; we denote by $\cO_X$ its adjoint orbit
$\cO_X=\{\Ad(g)X\mid g\in \bG(\K)\}$. Then $\cO_X$ (with the $p$-adic topology) is homeomorphic to
$\bG(\K)/C_G(X)$ (where $C_G(X)$ is the stabilizer of $X$);
{given Haar measures on $\bG(\K)$ and $C_G(X)$,} the space
$\bG(\K)/C_G(X)$ carries a $\bG(\K)$-invariant quotient measure.
For the fields $\K$ of characteristic zero,
it was proved by Deligne and Ranga Rao
\cite{ranga-rao:orbital} that when transported to the orbit of $X$,
this measure  is a Radon measure on $\fg(\K)$, i.e., it is finite on compact subsets of $\fg(\K)$ (strictly speaking, it is the group version of this statement that is proved
in \cite{ranga-rao:orbital}, but in characteristic zero this is
equivalent to the Lie algebra version). We denote this quotient
measure on $\bG(\K)/C_G(X)$ by $d^{\ast}g$;
then the \emph{orbital integral} at $X$ is
 the distribution $\Phi_X$ on $\tf(\fg(\K))$ defined by
$$
\Phi_X(f)=\int_{\bG(\K)/C_G(X)}f(\Ad(g) X)d^\ast g.$$
For the fields of good positive characteristic, convergence of this integral is proved by
McNinch \cite{mcninch:nilpotent}.
We emphasize that the orbital integral, as a distribution, depends on the normalization of Haar measures on $\bG(\K)$ and on $C_G(X)$, which together determine a normalization of the invariant measure on the orbit of $X$.
In this paper, instead of fixing these normalizations, we use the normalization of the measures on the orbits that comes
from a family of volume forms obtained from the symplectic forms on co-adjoint orbits, see
\S \ref{sub:noi}.

\subsubsection{Fourier transform}\label{sub:FT}
Given an additive  character $\Lambda$ of $\K$, we can define the Fourier transform
 on the Lie algebra $\fg(\K)$, which maps functions on $\fg(\K)$ to functions
on $\fg^\ast(\K)$.
\begin{defn}\cite{debacker:homogeneity}*{\S 3.1}
Let $dX$ be a Haar measure on $\fg(\K)$.
For any $f\in\tf(\fg(\K))$, let
\begin{equation*}
\hat f(\lambda)=\int_{\fg(\K)}f(X)\Lambda(\lambda(X))\, dX,
\end{equation*}
where $\lambda\in \fg^\ast(\K)$.
\end{defn}

The Fourier transform on $\fg^\ast(\K)$ is defined similarly.

\begin{rem} As pointed out in \cite{adler-debacker:mk-theory}*{\S 0},
there are in fact three objects appearing here: $\fg(\K)$, its linear dual
$\fg^\ast(\K)$, and  its Pontryagin dual $\widehat \fg(\K)$.
The choice of the character $\Lambda$
is equivalent to the choice of an identification of $\fg^\ast(\K)$ with
 $\widehat \fg(\K)$.
\end{rem}

 From now on, we will assume {that the characteristic of $\K$ is large enough so that Proposition  \ref{subsub:hy2} holds. }
Then one can use the bilinear form $\langle,\rangle$
from Proposition \ref{subsub:hy2}  to identify
$\fg(\K)$  with $\fg^\ast(\K)$.
With this identification, the definition of Fourier transform for a function
$f \in \tf(\fg(\K))$ takes the form:
\begin{equation*}
\hat f(Y)=\int_{\fg(\K)}f(X)\Lambda(\langle X, Y \rangle)\, dX,
\end{equation*}
and $\hat f$ is again a locally constant compactly supported
function on $\fg(\K)$.

With the identification of $\fg(\K)$ with $\fg^\ast(\K)$ given by the
form $\langle, \rangle$,
for a distribution $T$ on $\tf(\fg(\K))$,
its Fourier transform is defined to be
$$\widehat T (f)= T(\hat f).$$

\subsection{Local integrability theorems}\label{sub:HCtheorems}
When the field $\K$ has \emph{characteristic zero}, and $\bG$ is connected,
the following facts are due to Howe \cite{howe:fourier} and
Harish-Chandra \cite{hc:submersion}, \cite{hc:queens}:
\begin{enumerate}
\item\label{HC:character}
{For an admissible representation $\pi$ of $\bG(\K)$, its Harish-Chandra character
$\theta_\pi$ is a nice function on $\bG(\K)$; and in particular,
the representation of the distribution character (\ref{eq:rep.reg}) holds for
\emph{all} $f\in C_c^\infty(\bG(\K))$.}

\item\label{HC:orb-int} For an arbitrary
element $X\in \fg(\K)$ the Fourier transform of the orbital
integral $\Phi_X$ is represented by a
locally constant function $\widehat\mu_X$ supported on $\fg(\K)^\reg$:
$$\Phi_X(\widehat f)=\int_{\fg(\K)}f(g)\widehat\mu_X(g)\, dg
$$
for $f\in \tf(\fg(\K))$; and the function
$\widehat\mu_X$ {is nice.}

\end{enumerate}

Clozel \cite{clozel:nonconnected} extended these results
to the case of nonconnected $\bG$, still in characteristic zero.

In positive characteristic,
the existence of the locally constant function $\widehat\mu_X$  of (\ref{HC:orb-int}) on $\fg^\reg$, such that  the integral in (\ref{HC:orb-int}) converges for the test functions $f$
\emph{with support contained in $\bG(\K)^\reg$}  is proved by R. Huntsinger \cite{adler-debacker:mk-theory}*{Appendix A}, assuming the characteristic is large enough so that the orbital integrals are, indeed, distributions.

 For $\GL_n$, $\SL_n$, and their inner forms, {the statements (\ref{HC:character}) and (\ref{HC:orb-int}) }
 in positive characteristic were proved by Rodier \cite{rodier:loc-int} and Lemaire \cite{lemaire:gl_n}, \cite{lemaire:gl_n(D)}, \cite{lemaire:sl_n(D)}.
For general groups, we prove the analogues of these statements in this paper (as they have been out of reach up to now).

\subsection{Some spaces of distributions}\label{sub:inv.distr}
Everything in this short section
is quoted from \cite{debacker:homogeneity}.
Here we state the key result about the distributions with bounded support, which, in this precise
quantitative version and this generality is due to DeBacker. Recall the definitions first.

\begin{defn}
Let  $J(\fg(\K))$ denote the space of $\bG(\K)$-invariant distributions on
$\fg(\K)$, and
$J(\fg(\K)_r)$ denote the space of $\bG(\K)$-invariant distributions on
$\fg(\K)$ with support in $\fg(\K)_r$  {(where $\fg(\K)_r$ is the $G$-domain defined in
\S \ref{sub:MP} as a union of Moy-Prasad filtration lattices)}.
We use the similar notation $J(\fg^\ast(\K))$, $J(\fg^\ast(\K)_r)$
for the dual Lie algebra.
Let $J(\cN)$ denote the space of $\bG(\K)$-invariant distributions
whose support is contained in the set of nilpotent elements $\cN$.
Let $\cO(0, \K)$ denote the set of nilpotent orbits in $\fg(\K)$.
\end{defn}

\begin{defn}
Let $\cD_r$ be the space of functions on $\fg(\K)$ that can be represented as
a  {finite sum} $f=\sum f_i$, where $f_i$ a complex-valued, compactly supported function on
$\fg(\K)$, invariant under $\fg(\K)_{y_i, r}$ for some $y_i\in \scB(\bG, \K)$.
\end{defn}

\begin{rem}\label{rem:FTsupport}
We observe that with our choice of the conductor of the character $\Lambda$
 {(see \S \ref{sub:HM})},
if a test function $f$ on $\fg(\K)$ lies in the space
$\cD_r$, then the support of its Fourier transform $\hat f$ is contained
in $\fg(\K)_{(-r)^+}$, and if the support of $f$ is contained in $\fg(\K)_r$, then $\hat f \in \cD_{(-r)^+}$.
\end{rem}

The following statement is the summary of the part of the main result of
\cite{debacker:homogeneity} that  is used in this paper.

\begin{thm}\cite{debacker:homogeneity}*{Theorem 2.1.5, Corollary 3.4.6 and Remark 2.1.7}
\label{thm:distr}
Suppose all the hypotheses mentioned in \S \ref{sub:hypo} hold.
If $r\in \R$, then the distributions
$\{\res_{\cD_{r}}\Phi_\cO\}_{\cO\in \cO(0, \K)}$
form a basis of $\res_{\cD_{r}}J(\cN)$, and
$$\res_{\cD_{r}}J(\fg(\K)_{r})=\res_{\cD_{r}}J(\cN).$$
\end{thm}

\subsection{Local character expansion}
For an admissible representation $\pi$ of $\bG(\K)$, we denote its \emph{depth}
(defined in \cite{moy-prasad:k-types}*{Theorem 5.2})
by $\rho(\pi)$.

\begin{thm}(\cite{debacker:homogeneity}*{Theorem 3.5.2})\label{thm:expansion}
Let $\K$ be a
complete non-Archimedean local field with finite residue field of
characteristic $p$.
Let $\pi$ be an admissible
representation of ${\bG}(\K)$. Choose $r$ such that $\fg_r=\fg_{\rho(\pi)^+}$.
Suppose $p$ is sufficiently large so that the hypotheses
from \S \ref{sub:hypo} are satisfied. Suppose also that Hypothesis \ref{subsub:exp} holds. Then there exist constants
$c_{\cO}(\pi)\in \C$ indexed by $\cO(0, \K)$ such that
$$\theta_{\pi}(\mexp(X))=\sum_{\cO\in \cO(0, \K)}c_\cO(\pi)\fto(X)$$
for all $X\in \fg(\K)_r\cap \fg(\K)^\reg$.
\end{thm}

We observe that the coefficients $c_{\cO}(\pi)$ are defined only after the
field $\K$ is fixed; at present we do not have any general approach that
would yield information about the way they depend on the field,
since such an approach to begin with would require a field-independent way to
parameterise representations. For toral very supercuspidal representations
the beginnings of such a parameterisation are discussed in
\cite{cluckers-cunningham-gordon-spice}.

\subsection{Local character expansion near a tame semisimple element}\label{sub:loc_gamma}

Let $\bG$, $\K$, and  an admissible representation $\pi$ of $\bG(\K)$  be as above. For a semisimple element $\gamma\in \bG(\K)$,
its centralizer $C_G(\gamma)$ is a reductive (not necessarily connected) algebraic group
over $\K$. There is a finite list (depending only on the root datum of $\bG$) of the possible root data for the (connected components of) the centralizers of semisimple elements in $\bG(\K)$. We will denote a connected reductive group on this list by
${\bf M}^\circ$, and its Lie algebra by $\fm$.

Assume that the characteristic of $\K$ is large enough so that all the hypotheses of \ref{sub:hypo} hold for every possible $\bf M^\circ$ (the connected component of the centralizer of a semisimple element of $\bG$) in the place of
$\bG$. We also need to assume Hypothesis \ref{subsub:exp} for every such
$\bM^\circ$ (more precisely, we need the slightly weaker Hypothesis 8.5 from \cite{adler-korman:loc-char-exp}).
  We observe that when the characteristic of $\K$ is large enough, then both
$\bG^\circ$ and $\bM^\circ$ split over the same tame extension; thus, Hypothesis 8.3 of \cite{adler-korman:loc-char-exp} holds; therefore, the restriction of the mock
exponential map for $\bG(\K)$ satisfies the conditions of Hypothesis 8.5 from
\cite{adler-korman:loc-char-exp}.  {Thus, when the residue characteristic of $\K$ is large enough, it is sufficient to assume our Hypothesis
\ref{subsub:exp}, {for some $r>0$}.}

We need to introduce some more notation from \cite{adler-korman:loc-char-exp}.

Let $\gamma\in \bG^\sem(\K)$, and $C_G(\gamma)=\bM(\K)$ as above.
We can consider Moy-Prasad filtration subgroups  and the corresponding lattices in
$\fm$  {(as defined  in \S\ref{sub:MP}, with $\bM^\circ$ in place of $\bG$)}; so we have the subgroups $\bM^\circ(\K)_{x, r}$ for $x\in \cB(\bM^\circ, \K)$, etc.
Let $M_r=\bM(\K)_r$.
Following \cite{adler-korman:loc-char-exp}*{\S 4}, define, for $m\in M$:
$$D_{G/M}(m)= \det\left((\Ad(m)-1)\vert_{\fg/\fm}\right).$$
(with the convention that when $M=G$, $D_{G/M}\equiv 1$).
Further, for $r\ge 0$, let
\begin{equation*}
\begin{aligned}
&M_r'=\{m\in M_r\mid D_{G/M}(\gamma m)\neq 0\}  \\
&M_r''=\{ m\in M_r\mid \gamma m\in \bG(\K)^\reg\}.
\end{aligned}
\end{equation*}
Then $M_r''\subset M_r'$ are dense open subsets of $M_r$.

For an element $\gamma \in \bG(\K)^\sem$, Adler and Korman introduced the notion of
\emph{singular depth} $s(\gamma)$ (see \cite{adler-korman:loc-char-exp}*{Definition 4.1}); we will not need the precise definition here.
The main result we need is the following theorem (we are using our earlier notation $\theta_\pi$ for the function representing the distribution character of the representation $\pi$).

Let $\theta$ be the distribution on $M_r$ obtained from $\Theta_\pi$ via descent, as
explained in \cite{adler-korman:loc-char-exp}*{\S 7}. It is represented on $M_r''$ by
a locally constant function $\theta$, see \cite{adler-korman:loc-char-exp}*{Lemma 7.5}.
Then for $\theta$, an analogue of the local character expansion (in terms of the Fourier transforms of nilpotent orbital integrals on $\bM$) holds:

\begin{thm}\label{thm:expansion_ss}(\cite{adler-korman:loc-char-exp}*{Corollary 12.10}).
Let $r>\max\{\rho(\pi), 2s(\gamma)\}$.
Then
$$  {\theta}(\gamma \mexp(Y))=\sum_{\cO\in \cO_{\fm}} c_{\cO}\widehat\mu_{\cO}(Y)$$
for all $Y\in \fm_r'':=\mexp^{-1}(M_r'')$,
 {for some complex coefficients $c_{\cO}$ that depend on the representation $\pi$,
and where $\cO_{\fm}$ is the set of nilpotent orbits in $\fm$.}
\end{thm}

\section{Constructible exponential functions}\label{sec:mef}
Here we recall briefly the main notions and notation used in motivic
integration; we refer to the original articles \cite{cluckers-loeser},
\cite{cluckers-loeser:fourier}, \cite{cluckers-loeser:mixed} for complete details, and to \cite{cluckers-loeser:ax-kochen}, \cite{GY}, and especially \cite{cluckers-hales-loeser}
for exposition.

\subsection{Denef-Pas language and definable subassignments}\label{sub:DP}
Denef-Pas language is a first order language of logic designed for working with valued fields. The formulas in this language can have variables of three sorts:
the valued field sort, the residue field sort, and the value group sort (in our setting, the value group is always assumed to be $\Z$, so we call this sort the $\Z$-sort). Here is the list of symbols used
to denote operations and binary relations in this language:
\begin{itemize}
\item In the valued field sort: $+$ and $\times$ for the binary operations of
 addition and multiplication; $\ord(\cdot)$ for the valuation (it is a function from the valued field sort to the $\Z$-sort), and $\ac(\cdot)$ for the so-called angular component -- a function from the valued field sort to the residue field sort (more about this function below).
\item In the residue field sort: $+$ and $\times$ for addition and multiplication.
\item In the $\Z$-sort: $+$ for addition; the binary relations
$\ge$, and $\equiv_n$ for the congruence modulo $n$ for every $n\in \N$.
\item There is also the binary relation $=$ in every sort.
\end{itemize}

Initially, the symbols for the constants are just $0$ and $1$ in every sort, and the symbol $\infty$ in the $\Z$-sort to denote the valuation of $0$ (with the natural rules with respect to $\infty$ and all the operations and relations, such as $\infty\ge n$ is true for all $n$, etc.).

Given a number field $E$ with the ring of integers $\ri$, one can make a variant of Denef-Pas language with coefficients in $\ri\llb t \rrb$ in the valued
field sort. This means that a  {constant} symbol is formally added to the valued field sort for every element of $\Omega\llb t \rrb$.
We denote this language by $\cL_\Omega$. In this paper we use the language $\cL_\Z$;
however, since there might be applications where one wishes to work over a fixed number field $E$ that is different from $\Q$, here we discuss this slightly more general setting.

The formulas in $\cL_\Omega$ are built from the symbols for variables in every sort and constant symbols, using the listed above operations and
relations, and conjunction, disjunction, negation, and the quantifiers $\forall$ and $\exists$.

Given a valued field $\K$ that is an algebra over $\Omega$ \emph{with the choice of the uniformizer of the valuation $\varpi$}, one can interpret
the formulas in $\cL_\Omega$ by letting
the variables range, respectively, over $\K$, the residue field $\rf_\K$
of $\K$, and
$\Z$ (which is the value group of $\K$).
The function symbols $\ord(x)$ and $\ac(x)$ are interpreted as follows:
$\ord(x)$ denotes the valuation of $x$, and $\ac(x)$ denotes the so-called angular component of $x$: if $x$ is a unit, then $\ac(x)$ is the residue of $x$ modulo $\varpi$ (thus, an element of the residue field); for a general $x\neq 0$
define
$\ac(x)=\ac(\varpi^{-\ord(x)}x)$; thus, $\ac(x)$ is the first non-zero coefficient of the $\varpi$-adic expansion of $x$. By definition, $\ac(x)=0$ if and only if $x=0$.

In this way, any formula $\phi(x_1, \dots, x_n, y_1, \dots, y_m, z_1,
\dots z_r)$
with $n$ free (that is, not bound by
quantifiers) variables of the valued field sort, $m$ free variables of
the residue field sort, and $r$ free variables of the $\Z$-sort yields a
subset of $\K^n \times \rf_\K^m \times \Z^r$,
namely those points
$(x_1, \dots, x_n, y_1, \dots, y_m, z_1, \dots z_r)\in \K^n \times \rf_
\K^m \times \Z^r$ where $\phi$ takes the value ``true''. Sets of this
form for some $\cL_\Omega$-formula $\phi$
are called \emph{definable}. A function is called definable if its graph
is a definable set.

Let us (temporarily) denote the category of fields $L$ which admit an injective ring homomorphism  from
$\Omega$ to $L$ by
$\underline{\text{Flds}}_\Omega$. We write $h[n,m,r]$ for the functor
from $\underline{\text{Flds}}_\Omega$
to $\underline{\text{Sets}}$  {that} sends $L$ to $L \llp t \rrp ^n\times
L^m\times \Z^r$.
Any formula $\phi(x_1, \dots, x_n, y_1, \dots, y_m, z_1, \dots z_r)$ as
above in particular induces a map  {sending} any $L \in
\underline{\text{Flds}}_\Omega$ to a subset of $L \llp t \rrp ^n\times
L^m\times \Z^r$.
A map  {obtained} in this way from an $\cL_\Omega$-formula is
called a \emph{definable subassignment of
$h[n,m,r]$} (or simply a \emph{definable subassignment} if we do not
want to specify $n, m, r$).
A similar notion of assignments was first introduced in \cite{denef-loeser:p-adic}.

A morphism of definable subassignments consists of a family of maps
between the corresponding definable sets
for each $L \in \underline{\text{Flds}}_\Omega$, such that the family of graphs
of these maps is a definable subassignment.

\begin{defn}\label{def: def}
The category of definable (in the language $\cL_\ri$)
subassignments of $h[n,m, r]$ with some integers $n, m, r\ge 0$ is
denoted by $\de$.
The category of definable subassignments of $h[0,m,0]$ for some
$m>0 $ is denoted by $\rde$ (thus, the subassignments in $\rde$
are defined by formulas that can \emph{only} have free variables of the
residue field sort).
\end{defn}

We also need the ``relative'' situation: suppose $S\in \de$ is a
definable subassignment. Then one can define $\de_S$ -- the category of definable subassignments over $S$ -- to be the category of definable subassignments with a fixed morphism to $S$ (with morphisms, naturally, defined to be morphisms over $S$).
The category $\rde_S$ consists of subassignments of $S\times h[0,n,0]$
with the projection onto the first coordinate as
the fixed morphism to $S$.
If $X$ is a definable subassignment, we write $X[m,n,r]$ for $X\times h[m,n, r]$.

\subsection{Specialization}\label{sub:spec}
The main point of using the language $\cL_\ri$
is \emph{specialization}, which we survey briefly,  while referring to \cite{cluckers-loeser:ax-kochen}*{\S 6.7} or \cite{GY}*{\S 5} and \cite{cluckers-loeser:mixed} for a more extensive exposition.
{Let
$\cA$ be  the collection of completions of algebraic extensions of the
base field $E$,  and
let $\cB$ be the collection
of positive-characteristic  {local} fields that admit a homomorphism from $\ri$
 -- these are the collections
of fields to which we would like to apply a transfer principle.
Strictly speaking, we should write $\cA_\ri$ and $\cB_\ri$, but $\Omega$ is usually clear from the context; and in the main body of this paper $\Omega=\Z$, so we drop this subscript.}

Let $S$ be a definable subassignment of $h[n,m,r]$ for some $m$, $n$, and $r$; suppose that $S$ is defined by an $\cL_\ri$-formula $\phi$.
Let $\K\in \cA\cup \cB$ be a discretely valued field, with a
choice of the uniformizer of the valuation $\varpi$. Then the formula
$\phi$ can be interpreted in $\K$
 to give a subset $S_\K$ of $\K^n\times \rf_\K^m\times \Z^r$.
The set $S_\K$ is called the \emph{specialization} of the subassignment
$S$ to $\K$.

For two formulas $\phi_1$ and $\phi_2$ defining the same subassignment $S$,
there exists a constant $M$,
such that for $\K\in \cA_{M}\cup \cB_{M}$
their specializations to $\K$ give the same set regardless of which formula we use.
We emphasize that a definable subassignment can be specialized \emph{both}
to the fields
of characteristic zero and those of positive characteristic, and the specialization is well defined as long as the residue characteristic is sufficiently large.

\subsection{Motivic  exponential functions}\label{sub:cf}
For a definable subassignment $X$, the ring
of the so-called \emph{constructible motivic
functions} on $X$, denoted by $\cC(X)$, is defined in
\cite{cluckers-loeser}.  We will use a slight generalization of this, described below in 
\S \ref{subsub:gen}.
The elements of $\cC(X)$ are,
essentially, formal constructions defined using the language $\cL_\ri$.
For the sake of brevity (and consistency with \cite{TI}), we  drop the word
``constructible'' everywhere from now on, and refer to the elements of $\cC(X)$
as ``motivic functions''.
An important feature of
motivic functions is specialization to functions on definable subsets of
affine spaces over
discretely valued fields.
Namely, let $F\in \cC(X)$.
Let $\K\in \cA\cup \cB$ be a non-Archimedean local field.
Let $\varpi$ be the uniformizer of the
valuation on $\K$. Then
the motivic function $F$ specializes to a
$\Q$-valued function $F_\K$ on $X_\K$, for all fields
$\K$ of residue characteristic bigger than a constant that depends only on
  {the choice of} the
$\cL_\ri$-formulas defining $F$ and $X$.
As explained in \cite{cluckers-hales-loeser}*{\S 2.9}, one can tensor the
ring $\cC(X)$ with $\C$, and then the specializations $F_\K$ of
elements of $\cC(X)\otimes \C$ form a $\C$-algebra of functions on $X_\K$, which we denote by $\cC_\K(X_\K)$.
See \cite{TI}*{\S 4.2.5} for a general form of a motivic function.

Further, for a subassignment $X$ as above, the ring of motivic
constructible exponential functions
${\cC}^{\exp}(X)$ is defined
in \cite{cluckers-loeser:fourier}.
The elements of this ring specialize to what we call
($p$-adic) constructible exponential functions.
In the motivic setting, we also drop the word ``constructible'' from now on.
In order to get a specialization of a motivic exponential function, one needs to choose,
in addition to a local field $\K$  with uniformizer $\varpi$, an additive
character $\Lambda$ of $\K$ satisfying the condition
\begin{equation}\label{eq:D_K}
\Lambda(x)=\exp\left(\frac{2\pi i}{p}{\text{Tr}}_{\rf_\K}(\bar x)\right)
\end{equation}
for $x\in \ri_\K$.
Here, $p$ is the characteristic of $\rf_\K$,
$\bar x\in \rf_\K$ is the reduction of $x$ modulo $\varpi$, and
$\text{Tr}_{\rf_\K}$ is the trace of $\rf_\K$ over its prime subfield
(see \cite{TI}*{\S\S 4.1, 4.2.6} for details).
The set of characters of $\K$ satisfying the condition (\ref{eq:D_K})
is denoted by $\scD_\K$.

Given
a field $\K\in \cA\cup \cB$ as above,
with a uniformizer $\varpi$ and
an additive character $\Lambda$ as in (\ref{eq:D_K}), we consider the $\Q$-algebra
of functions on
$X_\K$ generated by the specializations
of motivic exponential functions. As above, we can tensor it with $\C$; we denote the resulting $\C$-algebra by
${\cC}^{\exp}_{\K, \Lambda}(X_\K)$.
See \cite{TI}*{\S 4.2.9, \S 3.2} for details.

We often need to talk about motivic (respectively,
motivic exponential) functions on the set of $\K$-points of an
algebraic group $\bG$ or its Lie algebra $\fg$. We observe
that any affine algebraic variety $V$
(for example, $V = \bG$ or $V = \fg$)
naturally gives a definable
subassignment of $h[m,0,0]$ with some $m$; let us for the moment denote this subassignment by $\tilde V$. Then $\tilde V_\K=V(\K)$,
for all non-Archimedean local
fields $\K$ of sufficiently large residue characteristic.
However, to keep notation simple, we simply talk about
motivic  functions on $V(\K)$ for a variety $V$, implying that we replace
$V(\K)$ with $\tilde V_\K$; it is in this sense that we talk about
motivic functions on
$\bG(\K)$ or $\fg(\K)$  {in this paper}.

In \cite{cluckers-loeser}, Cluckers and Loeser defined a class ${\mathrm I}C(X)$ of \emph{integrable}
motivic functions, which is closed under  integration with respect to
parameters (where integration is with
respect to the \emph{motivic measure}). Given a local field $\K$ with a choice
of the uniformizer, these functions specialize to
integrable (in the classical sense) functions on $X_\K$, and motivic
integration specializes to the
classical integration with respect to an appropriately normalized
Haar measure, when the
residue
characteristic of $\K$ is sufficiently large.
In \cite{cluckers-loeser:fourier} the definition of ``integrable'' and the notion of motivic integration are
extended to motivic \emph{exponential} functions.
Moreover, there is a notion of ``motivic'' Fourier transform that specializes to the classical Fourier transform.
In \cite{TI}, we provide a more general treatment of the issues of integrability, essentially, proving that any motivic exponential function whose specializations are $L^1$-integrable for almost all $\K\in \cA$ can be ``interpolated'' by a motivic exponential function 
integrable in the sense of 
\cite{cluckers-loeser:fourier} (where by ``interpolation'' we mean that it has the same specilaizations for every $\K\in \cA_M\cup \cB_M$ for a sufficiently large $M$), 
cf. \cite{TI}*{Theorem 4.3.3}. 

\subsubsection{A generalization of motivic exponential functions}\label{subsub:gen}
In some places in this paper, we have to take roots of $q$ (the cardinality of the residue field), e.g., in the notion of ``nice'' in \S \ref{sub:notation},
and in the definition of the measure on $\bG(\K)$ in \S \ref{sub:HM}. To this end, we
generalize the notion of motivic (exponential) functions as follows.
Let $F$ be a motivic (exponential) function on some $S$, let $f:S\to \Z$
be a definable morphism, and let $r \ge 1$ be an integer. We call any
expression $H$ of the form $F \lef^{\frac{1}{r}f}$ a motivic
(exponential) function on $S$, and we call the functions 
$F_\K q_\K^{\frac{1}{r}f_\K}$ on $S_\K$ the specializations $H_{\K}$ of $H$ for
$\K\in \cA_{M}\cup\cB_{M}$ of large residue field characteristic $q_\K$.
We also allow finite linear combinations of such expressions.
All classical results about motivic (exponential) functions easily
generalize to this setting, by splitting $S$ into $r$ disjoint parts
according to $f \mod r$.

\subsubsection{Conventions}\label{sub:convention}
For the sake of brevity, we use the term
``motivic (exponential) function'' a little loosely, in the sense that we  sometimes refer to a $p$-adic function
by this collection of adjectives if it is obtained by specialization from a
 motivic exponential function.
Precisely, we say that a function $f$ on some subset of an affine space over
a non-Archimedean local field $\K$ is  a
motivic (exponential) function if the following conditions hold:
\begin{enumerate}
\item the domain of $f$ is a specialization $S_\K$ of some definable subassignment $S\in \de$; and
\item there exists a motivic  (exponential) function $F$ on $S$,
and in the case when ``exponential'' is relevant, an additive character
$\Lambda\in \scD_\K$, such that
$f=F_\K$ (respectively,  $f=F_{\K, \Lambda}$).
\item   {If $\K$ is allowed to vary, then the definable subassignment $S$ and the motivic
(exponential) function $F$ can be taken independently of $\K$.}
\end{enumerate}
A similar convention applies to integration and Fourier transform; for example, when we integrate  the specialization of a motivic exponential
function (with respect to a $p$-adic Haar measure), we think of the integral as the specialization of the corresponding motivic integral.

Sometimes, we find it convenient to talk about families of motivic functions, of definable sets, or even of definable volume forms. We occasionally use the term ``constructible family'' to emphasize that the objects in question depend on the parameter indexing the family in a definable way. Thus, by a constructible family of motivic  {(exponential)} functions $\{f_a\}_{a\in S}$ on $X$, where $X$ and $S$ are definable subassignments, we mean nothing but a motivic
 {(exponential)} function on $S\times X$.

\subsection{Motivic exponential functions and representatives}\label{sub:repr}
As noted briefly in \S \ref{sub:spec} above, and explained in
\cite{TI} in detail, the specialization of a subassignment (and therefore,
of a motivic exponential function) depends, in principle, on the choice of specific formulas used to define the subassignment and the function in question. Given one such choice of formulas, there exists a constant $M>0$ such that
for the fields $\K\in \cA_{M}\cup\cB_{M}$, the specialization to $\K$ is well defined.  In \cite{TI}, the choice of formulas is referred to as ``the choice of representatives'', meaning that a subassignment is thought of as an equivalence class of formulas.

We observe that in this paper (as well as in all
applications of motivic integration so far) whenever we prove that a certain
object or function is ``motivic'', it automatically comes with a collection of
formulas defining it; that is, the motivic objects always appear with the choice of representatives in the sense of
\cite{TI}*{\S 4.2.2} (we emphasize again that the choice of representatives amounts to a choice of specific formulas defining the given subassignment).
Since all our definable objects
come with a choice of formulas defining them, we can assume that this is the
choice of representatives built into all the constants that provide the lower bounds on residue characteristic in all our results.

\subsection{Transfer of integrability and boundedness}
We quote the transfer of integrability and transfer of boundedness
principles from \cite{TI}.
(For simplicity, we quote the version without parameters, that is,
we take the parameterising space $X$ to be a point in \cite{TI}*{Theorem 4.4.1} and
\cite{TI}*{Theorem 4.4.2}).
\begin{thm}\cite{TI}*{Theorem 4.4.1}\label{thm:ti}
Let $F$ be a motivic  exponential
function on $h[n,0,0]$.
Then there exists $M>0$, such that for the
fields $\K\in\cA_{M}\cup\cB_{M}$,
the truth of the statement that $F_{\K,\Lambda}$
is (locally) integrable
for all $\Lambda \in \scD_\K$
depends only on the isomorphism class of the residue field of $\K$.
\end{thm}

\begin{thm}\cite{TI}*{Theorem 4.4.2}\label{thm:bounded}
Let $F$ be a motivic exponential
function on $h[n,0,0]$.
Then, for some $M>0$, for the fields $\K\in\cA_{M}\cup\cB_{M}$,
the truth of the statement that $F_{\K, \Lambda}$ is (locally) bounded
for all $\Lambda \in \scD_\K$
depends only on the isomorphism class of the residue field of $\K$.
\end{thm}

The main technical result of this paper is that Fourier transforms of
orbital integrals are represented on the set of regular elements by motivic exponential functions. Thus, the transfer principles apply, yielding local
integrability
(respectively, local boundedness) for
$\K\in \cB_{M}$ for large $M$.


\end{document}